\documentclass[a4paper,reqno]{amsart}

\usepackage{amsmath, amsfonts, amssymb, amsxtra, multicol, physics}

\usepackage[utf8]{inputenc}

\usepackage[english]{babel}
\usepackage{hyperref}
\usepackage[mathscr]{eucal}
\usepackage{tikz-cd}
\usepackage{comment}
\usepackage{enumitem}
\usepackage{seqsplit}

\hypersetup{colorlinks=true,linkcolor=red, anchorcolor=green, citecolor=cyan, urlcolor=red, filecolor=magenta, pdftoolbar=true}

\numberwithin{equation}{section}

\newtheorem{theorem}{Theorem}
\numberwithin{theorem}{section}
\newtheorem{proposition}[theorem]{Proposition}
\newtheorem{lemma}[theorem]{Lemma}
\newtheorem{corollary}[theorem]{Corollary}

\newtheorem{conjecture}[theorem]{Conjecture}

\theoremstyle{definition}\newtheorem{definition}[theorem]{Definition}
\theoremstyle{definition}\newtheorem{remark}[theorem]{Remark}
\theoremstyle{definition}\newtheorem{example}[theorem]{Example}
\theoremstyle{definition}\newtheorem*{notation*}{Notation}
\theoremstyle{definition}\newtheorem*{convention*}{Convention}
\theoremstyle{definition}\newtheorem{observation}{Observation}
\theoremstyle{definition}\newtheorem*{acknowledgment*}{Acknowledgments}

\newcommand{\N}{\mathbb{N}}
\newcommand{\R}{\mathbb{R}}
\newcommand{\C}{\mathbb{C}}
\newcommand{\A}{\mathcal{A}}
\newcommand{\h}{\mathcal{H}}
\renewcommand{\H}{\mathcal{H}}
\newcommand{\K}{\mathcal{K}}
\newcommand{\E}{\mathcal{E}}

\newcommand{\Q}{\mathbb{Q}}

\newcommand{\G}{\mathcal{G}}
\newcommand{\T}{\mathbb{T}}
\newcommand{\M}{\mathcal{M}}
\renewcommand{\L}{\mathcal{L}}
\newcommand{\hpi}{\mathcal{H}_{\pi}}
\newcommand{\B}{\mathcal{B}}
\newcommand{\bh}{\mathcal{B}(\mathcal{H})}

\newcommand{\bk}{\mathcal{B}(\mathcal{K})}
\newcommand{\bg}{\mathcal{B}(\mathcal{G})}
\newcommand{\sqd}{D^{1/2}}

\newcommand{\sha}{S_{\mathcal{H}}(\mathcal{A})}
\newcommand{\copa}{\pi(\mathcal{A})'}
\newcommand{\cst}{C^*}

\DeclareMathOperator{\id}{id}
\DeclareMathOperator{\Span}{span}
\DeclareMathOperator{\ospan}{\overline{span}}
\DeclareMathOperator{\ran}{\mathscr{R}}

\DeclareMathOperator{\Lat}{Lat}
\DeclareMathOperator{\Alg}{Alg}

\begin{document}

\title{$C^\ast$-extreme maps and nests}

\author{B.V. Rajarama Bhat}
\email{bhat@isibang.ac.in}

\author{Manish Kumar}
\email{manish\_rs@isibang.ac.in}

\address{Indian Statistical Institute, Stat-Math. Unit, R V College Post, Bengaluru 560059, India}

\begin{abstract}
    The generalized state space $ S_{\mathcal{H}}(\mathcal{\mathcal{A}})$ of all unital completely positive (UCP) maps on a unital $C^*$-algebra $\mathcal{A}$ taking values in the algebra $\mathcal{B}(\mathcal{H})$ of all bounded operators on a  Hilbert space $\mathcal{H}$, is a $C^\ast$-convex set.  In this paper, we establish a connection between $C^\ast$-extreme points of $S_{\mathcal{H}}(\mathcal{A})$ and a factorization property of certain algebras associated to the UCP maps. In particular, the factorization property of some nest algebras is used to give a complete characterization of those $C^\ast$-extreme maps which are direct sums of pure UCP maps. This significantly extends a result of Farenick and Zhou [Proc. Amer. Math. Soc. 126 (1998)]  from finite to infinite dimensional Hilbert spaces. Also it is shown that  normal $C^\ast$-extreme maps on type $I$ factors are direct sums of normal pure UCP maps if and only if an associated algebra is reflexive. Further, a Krein-Milman type theorem is established for $C^\ast$-convexity of the set $ S_{\mathcal{H}}(\mathcal{A})$ equipped with bounded weak topology, whenever $\mathcal{A}$ is a separable $C^\ast$-algebra or it is a type $I$ factor. As an application, we provide a new proof of a classical  factorization result on operator-valued Hardy algebras.
\end{abstract}



\maketitle

\section{Introduction}\label{section:introduction}

Quantization in functional analysis and search for the right
noncommutative analogue of various classical concepts have
attracted considerable amount of interest among operator
algebraists.  Several different notions of quantizations of
convexity have appeared in the literature over the decades, among
which we cite a few \cite{DaKe, EfWi, FaMo, FaPlSm, Fu, LoPa,  Ma2}.
One such natural extension
is $\cst$-convexity, where the idea is to replace scalar valued convex
coefficients by $\cst$-algebra valued coefficients. This particular
notion has been explored in different contexts. The initial
definition was for subsets of $\cst$-algebras \cite{LoPa}.
Subsequently it has been  extended on similar lines for subsets of
bimodules over $\cst$-algebras \cite{Ma2}, for spaces of unital
completely positive maps \cite{FaMo}, and for  positive operator
valued measures \cite{FaPlSm}. In all these frame-works,  one of the
primary goals has been to identify $\cst$-extreme points of the
corresponding $\cst$-convex sets and look for an analogue of
Krein-Milman theorem. Our focus in this paper is the $\cst$-convexity structure of
the generalized state space $\sha$ of all unital completely positive
(UCP) linear maps from a unital $\cst$-algebra $\A$ to $\bh$, the
algebra of all bounded operators on a separable Hilbert space $\h$.
Generalized state spaces are thought of as quantizations of usual
state spaces.

Motivated by  the ideas of Loebl and Paulsen \cite{LoPa}, the notion
of $\cst$-convexity and $\cst$-extreme points of $\sha$ was  defined and  studied by
Farenick and Morenz \cite{FaMo}.
They  developed some general
properties, however the main focus remained on the case when $\h$ is a finite
dimensional Hilbert space,  that is the case, $\h = \C ^n$ for some
$n\in \N.$ They gave a complete description of $\cst$-extreme points
of $ S_{\C^n}(\A)$, whenever $\A$ is a commutative $\cst$-algebra or
a finite dimensional matrix algebra.
Following this work, Farenick and Zhou \cite{FaZh} came up with an
abstract characterization of  $\cst$-extreme points via Stinespring
decomposition, using which the structure of all $\cst$-extreme
points of $S_{\C^n}(\A)$ was illustrated for  an arbitrary
$\cst$-algebra $\A$.  It was shown  that all such
maps are direct sums of pure UCP maps satisfying some `nested'
properties.

In the case when the $\cst$-algebra $\A$ is commutative and the
Hilbert space $\h$  is arbitrary dimensional, the techniques of
positive operator valued measures  were exploited by Gregg \cite{Gr}
to study  necessary conditions for $\cst$-extreme points of $\sha$.
Banerjee et al. \cite{BaBhMa} have recently used this  approach to
show in particular that all $\cst$-extreme points of $\sha$ are
$*$-homomorphisms whenever $\A$ is a commutative $\cst$-algebra with
countable spectrum. The purpose of this article is to undertake a systematic investigation of the structure  of $\cst$-extreme points in $\sha$ for arbitrary (not-necessarily commutative) $\cst$-algebras $\A .$

We begin with a
discussion of some abstract characterizations of $\cst$-extreme
points of $\sha$ in Section \ref{section:general properties}. Making
use of a result from \cite{FaZh}, we present a connection between
$\cst$-extreme points and factorization property of certain
subalgebras  in some von Neumann algebras (see Definition
\ref{definition of algebras with factorization} and Corollary \ref{a
factorization property of algebras coming out of C* extreme point}).
  Our
main insight is that characterizations of $C^*$-extreme points
naturally lead to some nests of subspaces and we can invoke some
well-known factorization results of the associated nest algebras. In
other words we develop  a strong mathematical link between the
theory of $C^*$-extreme points of UCP maps and that of nest
algebras.
 The theory of nest algebras
and their factorization property has a long history, for which we
refer the readers to the beautiful book by Davidson \cite{Da}.

One of our main results is Theorem \ref{direct sum of pure UCP maps}
in Section \ref{section:direct sum of pure UCP maps}, which
generalizes significantly a result of \cite{FaZh} from finite
dimensional Hilbert spaces  to infinite dimensions. More precisely,
a complete description is given for countable direct sums of  pure
UCP maps to be $\cst$-extreme points. This result pinpoints  a
refinement needed to a sufficiency condition suggested by
\cite{FaZh} for $\cst$-extremity of direct sums of pure UCP maps.

 Section \ref{section:normal UCP maps} is devoted to the study of normal
$\cst$-extreme maps on  type $I $ factors.   The structure of normal
UCP maps on these algebras is well-known.  This knowledge helps us
to arrive at necessary and sufficient conditions for normal
$\cst$-extreme  maps to be direct sums of normal pure UCP maps
(Theorem \ref{normal C*-extreme maps and iff criteria}). In the
course of the proof, we apply a fact recently proved by the authors
\cite{BhMa} that all reflexive algebras having factorization are
nest algebras.

A fundamental result in classical convexity theory is Krein-Milman theorem for compact convex sets in locally convex topological vector spaces.
Naturally, an analogue of Krein-Milman theorem is  expected  for
quantized convexity under appropriate topology. Several researchers
have been quite successful in reaching this goal under varying
set-ups, particularly when the operator-valued coefficients are taken from finite dimensional $\cst$-algebras: see for example, for compact $\cst$-convex subsets of $M_n$ \cite{Mo}, for
compact matrix convex sets in locally convex spaces \cite{WeWi}, and
for weak$^*$-compact $\cst$-convex sets in hyperfinite factors \cite{Ma}.
However there are instances where such theorems fail to hold.  In fact Magajna \cite{Ma2} produced an example of a weak$^*$-compact $\cst$-convex subset of an operator $\B$-bimodule over a commutative von Neumann algebra $\B$ which does not even possess any $\cst$-extreme point.
Nevertheless, for  $\cst$-convex spaces of UCP maps equipped with
bounded weak topology, some promising  results have appeared in
restricted cases. More specifically,    Krein-Milman type theorems are
known to be true  for $\cst$-convexity of the space $S_{\h}(\A)$ in
the following two cases: (1)  when $\A$ is an arbitrary
$\cst$-algebra and $\h$ a finite dimensional Hilbert space
\cite{FaMo}, (2) when $\A$ is a commutative $\cst$-algebra  and $\h$
has arbitrary dimension \cite{BaBhMa}.  We extend this line of
research in Section \ref{section:Krein-Milman theorem}, by showing a
Krein-Milman type theorem for $\cst$-convexity of $\sha$, whenever
$\h$ is infinite dimensional and separable, and $\A$ is a separable
$\cst$-algebra or a type $I$ factor (Theorem \ref{Krein-Milman type
theorem}). Whether the same holds for $\sha$ in full generality
remains as an open question.

Finally in Section \ref{section:examples and application}, we produce a number of examples of $\cst$-extreme maps, and consider their applications. At first, behaviour of $\cst$-extreme points  under minimal tensor product of UCP maps are examined to derive more $\cst$-extreme points.
Further, examples of certain $\cst$-extreme points in $ S_\h(C(\T))$ are seen (here $C(\T)$ is the space of all continuous functions on the unit circle $\T$), using which we provide a new proof of a known classical result of Szeg\" o and its operator valued analogue about factorization property of operator valued Hardy algebras. Lastly factorization property of some well-known algebras are utilized to produce  examples of  UCP maps, some of which are $\cst$-extreme  and some  are not.

The following convention will be followed throughout the paper.
All Hilbert spaces on which  completely positive maps act  are complex and separable, where the inner product is assumed to be linear in the second variable.
For Hilbert spaces $\h$ and $\K$, $\B(\h,\K)$ denotes the space of all bounded linear operators from $\h$ to $\K$.
We denote by $\bh$ the algebra of all bounded operators on $\h$.  By subspaces, projections and operators, we mean closed subspaces, orthogonal projections and bounded operators respectively. For any subset $E$ of $\h$, $[E]$ denotes the closed  subspace generated by $E$.  The orthogonal complement of a subspace $F$ in a subspace  $E $ will be denoted by  $E\ominus F$.  If $\{E_i\}_{i\in\Lambda}$ is a collection of subspaces of $\h$, then we write $\wedge_{i\in\Lambda}E_i=\cap_{i\in\Lambda}E_i$ and  $\vee_{i\in\Lambda}E_i=[\cup_{i\in\Lambda}E_i]$.  For any subspace $E$, we denote by $P_E$ the projection onto  $E$. All $\cst$-algebras considered will be assumed to contain identity, which we  denote by $1$ or by $I_\h$ if the Hilbert space $\h$ on which the algebra acts  needs to be specified.
For any self-adjoint subalgebra $\M$ of $\B(\h)$, we denote by $\M'$ the commutant of $\M$ in $\bh$.

\section{General properties of $C^\ast$-extreme points}\label{section:general properties}

We begin with some preliminaries on the theory of completely
positive maps and their dilations. Let $\A$ be a unital
$\cst$-algebra, and $\h$  a separable Hilbert space. A linear map
$\phi:\A\to\bh$ is called {\em positive} if $\phi(a)\geq0$ in $\bh$
whenever $a\geq 0$ in $\A$, and $\phi$ is called {\em completely
positive (CP)} if $\phi\otimes \id_n:\A\otimes M_n\to\B(\h)\otimes
M_n$ is positive for all $n\geq 1$. Here $\id_n$ is the identity map
on the $\cst$-algebra $M_n$ of $n\times n$ complex matrices. A
unital $*$-homomorphism  from $\A$ to $\bh$ is called a {\em
representation}. Note that a map of the form $a\mapsto V^*\pi(a)V$,
$a\in\A$, for a representation $\pi$ on $\A$ and an appropriate
operator $V$, is a CP map on $\A$.

Conversely, the well-known Stinespring dilation theorem says: if $\phi:\A\to\bh$ is a CP map, then there is a triple $(\pi,V,\hpi)$ of  a Hilbert space $\hpi$, an operator $V\in\B(\h,\hpi)$ and a representation $\pi:\A\to\B(\hpi)$ such that $\phi(a)=V^*\pi(a)V$ for all $a\in\A$, and satisfies the minimality condition that $\hpi=[\pi(\A)V\h]$.
Moreover, any such triple is unique upto unitary equivalence. We call $(\pi,V,\hpi)$ the {\em minimal Stinespring triple} for $\phi$. Note that $V$ is an isometry if and only if $\phi$ is unital (i.e. $\phi(1)=I_\h$).

We remark here that although the Hilbert space  on  which a CP map acts is  assumed to be separable,  the Hilbert space $\hpi$ in the minimal Stinespring triple $(\pi,V,\hpi)$ may not be separable. However, when the $\cst$-algebra $\A$ is also separable,  $\hpi$ is separable. See \cite{Ar69, Pa, Pi} for more details on the theory of  CP maps. We fix the following notation for the rest of the article.

\begin{notation*}
We denote by $\sha$  the collection of all unital completely positive (UCP) maps from a unital  $\cst$-algebra $\A$ to $\bh$.
\end{notation*}

The set $\sha$ is called {\em generalized state space} on the
$\cst$-algebra $\A$. Note that $S_\C(\A)$ is the usual state space
of $\A$. The set $\sha$ possesses  both linear as well as other
quantized convexity structure. In particular, the $\cst$-convexity
structure of $\sha$ has played a very important role in understanding
general theory of completely positive maps and various concepts
associated with them, and this is the main theme of this paper.

We  list two very important theorems on completely positive maps proved by Arveson \cite{Ar69}. For any two completely positive maps $\phi,\psi:\A\to\bh$, we say $\psi\leq \phi$, if $\phi-\psi$ is completely positive. Below we state a Radon-Nikodym type theorem (Theorem 1.4.2, \cite{Ar69}) for comparison of two completely positive maps.

\begin{theorem}[Radon-Nikodym type Theorem]\label{Radon-nikodym type theorem}
Let $\phi:\A\to\bh$ be a completely positive map with minimal Stinespring triple $(\pi,V,\hpi)$. Then a completely positive map $\psi:\A\to\bh$ satisfies $\psi\leq \phi$ if and only if there is a positive contraction $T\in\copa$ such that $\psi(a)=V^*T\pi(a)V$ for all $a\in\A$.
\end{theorem}

Although we are mainly concerned about
$\cst$-extreme points of UCP maps, we sometimes mention results
about (linear) extreme points as well, in order to make comparisons
between the two situations. Clearly, the set $\sha$ is a convex set
(i.e. $\sum_{i=1}^n \lambda_i\phi_i\in\sha$, whenever
$\phi_i\in\sha$ and $\lambda_i\in[0,1]$, $1\leq i\leq n$ with
$\sum_{i=1}^n \lambda_i=1$). The following is an abstract characterization of extreme points of UCP maps due to Arveson (Theorem 1.4.6, \cite{Ar69}).

\begin{theorem}[Extreme point condition]\label{Extreme point condition}
Let $\phi\in\sha$, and let $(\pi,V,\hpi)$ be its minimal Stinespring triple. Then $\phi$ is extreme in $\sha$  if and only if
 the map $T\mapsto V^*TV$ from $\pi(\A)'$ to $\bh$ is injective.
\end{theorem}

We now  turn our attention to the main topic of $\cst$-convexity of the generalized state space $\sha$. The  space $\sha$ is a {\em $\cst$-convex set} in the following sense: If $\phi_i\in\sha$ and $T_i\in\bh$ for $1\leq i\leq n$ with $\sum_{i=1}^nT_i^*T_i=I_\h$, then their {\em $\cst$-convex combination}
\begin{equation*}
\phi (\cdot ):=  \sum_{i=1}^nT_i^*\phi_i(\cdot)T_i
\end{equation*}
is in $\sha$. The operators $T_i$'s are called {\em
$\cst$-coefficients}. When $T_i$'s are invertible, the sum is  called
a {\em proper $\cst$-convex combination} of $\phi$. Following
\cite{FaMo}, we consider the following definition:

\begin{definition}\label{definition of C*-extreme points}
A UCP map $\phi:\A\to\bh$ is  called a {\em $\cst$-extreme point} of $\sha$ if whenever
\[\phi(\cdot)=\sum_{i=1}^nT_i^*\phi_i(\cdot)T_i,\]
is a proper $\cst$-convex combination of $\phi$, then  $\phi_i$ is unitarily equivalent to $\phi$ for each $i$ i.e. there is a unitary $U_i\in\bh$ such that $\phi_i=U_i^*\phi(\cdot)U_i$.
\end{definition}

It is clear that every map unitarily equivalent to a $\cst$-extreme point is also $\cst$-extreme.   The structure of $\cst$-extreme points of $\sha$ has been studied extensively, see \cite{BaBhMa, FaMo, FaPlSm, FaZh, Gr, Ma2, Zh} among others. The aim of this article is to understand the behaviour of $\cst$-extreme points of $\sha$, upto unitary equivalence.

A key ingredient in our approach is a result by Farenick and Zhou \cite{FaZh}, who  taking cue from Arveson's extreme
point condition for UCP maps provided an abstract characterization of
$\cst$-extreme points of $\sha$ by making use of Stinespring
decomposition.
We restate their result with minor modifications in our notation and
give an outline of the proof. In what follows, $\ran(T)$ denotes the
range of an operator $T$.

\begin{theorem}[Theorem 3.1, \cite{FaZh}]\label{Farenick Zhou criterion}
Let $\phi:\A\to\bh$ be a UCP map with the minimal Stinespring triple $(\pi,V,\hpi)$. Then $\phi$ is $C^*$-extreme in $\sha$ if and only if for any positive operator $D \in\copa$ with $V^*DV$  invertible, there exist a partial isometry $U\in\copa$ with $\ran(U^*)=\ran(U^*U)=\overline{\ran(D^{1/2})}$ and an invertible $Z\in\bh$ such that $UD^{1/2}V=VZ$.
\end{theorem}
\begin{proof}
$\Longrightarrow$ Let $\phi$ be $C^*$-extreme in $\sha$, and let $D\in\copa$ be positive with $V^*DV$  invertible. Choose $\alpha>0$  small enough so that $I_{\hpi}-\alpha D$ is positive and invertible. Set $ T_1=(\alpha V^*DV)^\frac{1}{2}$ and $T_2=(V^*(I_{\hpi}-\alpha D)V)^{1/2}$. Then $T_1,T_2$ are invertible, and satisfy $T_1^*T_1+T_2^*T_2=V^*V=I_\h$. Now we define $\phi_1,\phi_2:\A\to\bh$ by
\[\phi_1(a)=T_1^{-1}(\alpha V^*D\pi(a)V) T_1^{-1},~ \mbox{ and }~ \phi_2(a)=T_2^{-1}V^*(I_{\hpi}-\alpha D)\pi(a)VT_2^{-1}
\]
for all $a\in\A$. Clearly  $\phi_1$ and $\phi_2$ are UCP maps such that
$\phi(a)=T_1^*\phi_1(a)T_1+T_2^*\phi_2(a)T_2$, $a\in\A$.
Since $\phi$ is a $C^*$-extreme point in $\sha$, there exists a unitary $W\in\bh$ such that for all $a\in\A$, we have $\phi(a)=W^*\phi_1(a)W$, that is,
\begin{equation*}
    \phi(a)=(\sqrt{\alpha}D^{1/2}VT_1^{-1}W)^*\pi(a)(\sqrt{\alpha}D^{1/2}VT_1^{-1}W)=X^*\pi(a)X,
\end{equation*}
where $X=\sqrt{\alpha}D^{1/2}VT_1^{-1}W$. It is easy to verify  that $[\pi(\A)X(\h)]=\overline{\ran(D^{1/2})} $ (call it $\K$). Then the
triple $(\pi(\cdot)_{|{\K}},X,\K)$ is another minimal Stinespring triple for $\phi$; hence 
by uniqueness, there exists a unitary operator $\widetilde{U}:\K\to \hpi$ such that \begin{equation*}
    \widetilde{U}X=V,\quad\text{ and }\quad \pi(a)\widetilde{U}=\widetilde{U}\pi(a)_{|_{\K}}\quad \mbox{for all}~ a\in\A.
\end{equation*}
Extend $\widetilde{U}$ to $\hpi$ by assigning $0$ on the complement of $\K$, and call this map $U$. Then $U$ is a partial isometry (in fact, a co-isometry) with ${\ran}(U^*)=\K$.
We also note that $UX=V$ and  $\pi(a)U=U\pi(a)$ for all $a\in\A$, so $ U\in\copa$. Further, we have
\begin{equation*}
    V=UX=U\sqrt{\alpha}D^{1/2}VT_1^{-1}W=U\sqd VZ^{-1}
\end{equation*}
where $Z=\frac{1}{\sqrt{\alpha}}W^*T_1\in\bh$ is invertible; hence we get $U\sqd V=VZ$.\\
$\Longleftarrow$ Assume the `only if' condition, and  let $\phi(\cdot)=\sum_{i=1}^n T_i^*\phi_i(\cdot)T_i$ be a proper $C^*$-convex combination of $\phi$. Then   $T_i^*\phi_i(\cdot)T_i\leq \phi(\cdot)$ for each $i$, so by Radon-Nikodym type theorem (Theorem \ref{Radon-nikodym type theorem}) there exists $D_i\in\copa$ with $0\leq D_i\leq I_{\hpi}$ such that $T_i^*\phi_i(\cdot)T_i=V^*D_i\pi(\cdot)V$. Note that $V^*D_iV=T_i^*T_i$, so $V^*D_iV$ is invertible; hence by  hypothesis, there exist a partial isometry $U_i\in\copa$ with $\ran(U_i^*U_i)=\overline{\ran(D_i^{1/2})}$ and an invertible $Z_i\in\bh$ such that $U_iD_i^{1/2}V=VZ_i$. Note that  $U_i^*U_iD_i^{1/2}=D_i^{1/2}$; hence for all $a\in\A$, we have
\begin{equation*}
\begin{split}
    T_i^*\phi_i(a)T_i&=V^*D_i\pi(a)V=V^*D_i^{1/2}\pi(a)D_i^{1/2}V
    =V^*D_i^{1/2}\pi(a)U_i^*U_iD_i^{1/2}V\\
    &=(U_iD_i^{1/2}V)^*\pi(a)(U_iD_i^{1/2}V)
    =(VZ_i)^*\pi(a)(VZ_i)
    =Z_i^*\phi(a)Z_i,
\end{split}
\end{equation*}
which in other words says $\phi_i(a)=W_i^*\phi(a)W_i$, where $W_i=Z_iT_i^{-1}$. Note that $W_i^*W_i=\phi_i(1)=I_\h$,  and since $W_i$ is invertible, it follows that  $W_i$ is unitary. Thus $\phi_i$ is unitarily equivalent to $\phi$ for each $i$, which concludes that $\phi$ is a $C^*$-extreme point in $\sha$.
\end{proof}

It is claimed in  \cite{FaZh}  that the operator $U$ in the statement of
 Theorem \ref{Farenick Zhou criterion} above is a unitary. At this point, we do not know whether $U$ can be chosen to be a unitary.

The following corollary is a characterization of $\cst$-extreme  maps provided by Zhou \cite{Zh}. The proof follows directly
from Theorem \ref{Farenick Zhou criterion} and Radon-Nikodym type theorem. However, the statement as written in \cite{Zh} has a minor error; see Example 3.7 in \cite{BaBhMa} for a counterexample (which is stated there in the language of positive operator valued measures). Also see Example \ref{invertibility cannot be dropped} below.  The proof of
the following proceeds on almost the same lines as in \cite{Zh}, so it is left to the readers.

\begin{corollary}[Theorem 3.1.5, \cite{Zh}]\label{Zhou criteria for C^* extreme points}
Let $\phi\in\sha$. Then $\phi$ is $\cst$-extreme in $\sha$ if and only if for any completely positive map $\psi$ satisfying $\psi\leq \phi$ with $\psi(1)$ invertible,  there exists an invertible operator $T\in\bh$ such that $\psi(a)=T^*\phi(a)T$ for all $a\in \A$.
\end{corollary}

We now give another abstract characterization of $\cst$-extreme points, whose proof follows from a direct application of Theorem \ref{Farenick Zhou criterion} and polar decomposition of operators. This powerful characterization  turns out to be the most useful for our purpose.

\begin{corollary}\label{a C*-extreme criterian in language of factorization of positive operators}
Let $\phi:\A\to\bh$ be a UCP map with minimal Stinespring triple  $(\pi, V,\h_{\pi})$. Then $\phi$ is  $C^*$-extreme  in $\sha$ if and only if for any positive operator $D\in\copa$ with $V^*DV$  invertible, there exists $S\in\copa$ such that $D=S^*S$, $ SVV^*=VV^*SVV^*$  and $V^*SV$ is invertible (i.e. $S(V\h)\subseteq V\h$ and $S_{|_{V\h}}$ is invertible).
\end{corollary}
\begin{proof}
$\Longrightarrow$ We use the equivalent conditions for $\cst$-extreme points as in  Theorem \ref{Farenick Zhou criterion}.  Assume first that $\phi$ is a $C^*$-extreme point in $\sha$. Let $D\in\copa$ be a positive operator such that $V^*DV$ is invertible.   By Theorem \ref{Farenick Zhou criterion}, there exist a partial isometry $U\in\copa$ with $U^*UD^{1/2}=D^{1/2}$ and an invertible $Z\in\bh$ such that $UD^{1/2}V=VZ.$ Set $S=UD^{1/2}$. Then $S^*S=D^{1/2}U^*UD^{1/2}=D$ and $V^*SV=V^*UD^{1/2}V=V^*VZ=Z$. Thus $V^*SV$ is invertible, and we get
\begin{equation*}
    SVV^*=UD^{1/2}VV^*
    =(VZ)V^*=VV^*(VZ)V^*
    =VV^*(UD^{1/2}V)V^*
    =VV^*SVV^*.
\end{equation*}
$\Longleftarrow$ Assume the `only if' conditions. To show that $\phi$ is $\cst$-extreme in $\sha$, let $D\in\copa$ be positive  with $V^*DV$ invertible. By hypothesis, there exists $S\in\copa$ such that $D=S^*S$, $SVV^*=VV^*SVV^*$ and $V^*SV$ is invertible. Let $S=U\sqd$ be the polar decomposition of $S$, where $U$ is a partial isometry with initial space $\overline{\ran(\sqd)}$ i.e. $\ran(U^*)=\overline{\ran(D^{1/2})}$. Since $S\in\copa$, and $\copa$ is a von Neumann algebra, it follows  that $U\in\copa$. Further, we have
\begin{equation*}
    U\sqd V=SV=(SVV^*)V=(VV^*SVV^*)V=VV^*SV=VZ,
\end{equation*}
where $Z=V^*SV\in\bh$, which is invertible. That $\phi$ is $\cst$-extreme in $\sha$ now follows from the equivalent criteria of  Theorem \ref{Farenick Zhou criterion}. This completes the proof.
\end{proof}

In the corollary above, we cannot drop the assumption that $V^*DV$ is invertible as the following example shows. Below $\T$ is the unit circle with one dimensional Lebesgue measure, $C(\T)$ is the space of continuous functions on $\T$, and $ H^2= H^2(\T)$ is the Hardy space.

\begin{example}\label{invertibility cannot be dropped}
Consider the UCP  map $\phi:C(\T)\to \B( H^2)$ defined by
\begin{equation}\label{eq:Hardy space example}
    \phi(f)=P_{ H^2}{M_f}_{|_{ H^2}}=T_f~~\mbox{for all }f\in C(\T).
\end{equation}
Here $M_f$ is the multiplication operator on $L^2(\T)$ by the symbol $f$.
Then $\phi$ is a $\cst$-extreme point in $ S_{ H^2}(C(\T))$ (Example 2, \cite{FaMo}). Note that $\phi$ is already in minimal Stinespring form with the representation
$\pi:C(\T)\to \B(L^2(\T))$  given by $\pi(f)=M_f$.
Then it is well-known that $\pi(C(\T))'=\{M_f; f\in L^\infty(\T)\}\subseteq\B(L^2(\T))$. Now let $d\in L^\infty(\T)$ be  such that $d\geq 0$ a.e. and  the subset $\{x\in \T; d(x)=0\}$ has positive one-dimensional Lebesgue measure.  It is then clear that $M_d$ is  not invertible which is equivalent to saying that $P_{ H^2}{M_d}_{|_{ H^2}}$  is not invertible.  Now let if possible, there exists $s\in L^\infty(\T)$ such that $d=  \bar{s}s$ and $M_s( H^2)\subseteq H^2$. This implies that $s\in  H^\infty(\T)$. But then   the zero set of any  function in $H^\infty(\T)$ (in particular, $s$) has zero  measure (Theorem 25.3, \cite{Co}). This contradicts the assumption that zero set of the function $d$ has positive  measure.
\end{example}

\begin{observation}\label{observation}
Let $\phi$ be a $\cst$-extreme point in $\sha$ with minimal Stinespring triple $(\pi,V,\hpi)$. Then for any positive $D\in\copa$ with $V^*DV$ invertible, we observe the following from the proof of Theorem \ref{Farenick Zhou criterion}:
\begin{itemize}
    \item There is a co-isometry $U$ with $\ran(U^*)=\overline{\ran(D^{1/2})}$ and an invertible $Z$ such that $UD^{1/2}V=VZ$. In particular if $D$ is one-one (equivalently, $D$ has dense range), then  $U$ is unitary.
    \item If $S=UD^{1/2}$, then $S^*$ is one-one.
    \item Also $V^*SV$ is invertible such that $\|(V^*SV)^{-1}\|^2=\|(V^*DV)^{-1}\|$.
\end{itemize}
\end{observation}

We digress momentarily from $\cst$-convexity and consider a notion of factorization  of  subalgebras (not necessarily self-adjoint) in $\cst$-algebras. We shall return to $\cst$-extreme maps by providing their connection with such algebras.

\begin{definition}\label{definition of algebras with factorization}
A subalgebra $\M$ of a $\cst$-algebra $\A$ has  {\em factorization} in   $\A$ if for any positive and invertible element $D\in\A$, there is an invertible element $S$ such that $S,S^{-1}\in\M$ and $D=S^*S$.
\end{definition}

 The following proposition  follows directly from the definition of factorization and $*$-closed property of
 $C^*$-algebras, and so we omit the proof.
  Here and elsewhere, $\mathcal{S}^*$ denotes the set $\{S^*; S\in\mathcal{S}\}$ for any subset $\mathcal{S}$ of a $\cst$-algebra $\A$.

\begin{proposition}\label{M* has factorization}
If a subalgebra $\M$ has factorization in a $\cst$-algebra $\A$, then $\M^*$ also has factorization in $\A$ i.e. for any positive and invertible element $D\in\A$, there is an invertible element $S\in\M$ with $S,S^{-1}\in\M$ such that $D=SS^*$.
\end{proposition}

The factorization property of several non-selfadjoint algebras has widely been studied, of which we mention a few. The Cholesky factorization theorem talks about  factorization property of the algebra of upper triangular matrices  in $M_n$, the algebra of all $n\times n$ matrices. A result of Szeg\"o says that the Hardy algebra $H^\infty(\T)$ on the unit circle has factorization in $L^\infty(\T)$ (see Corollary \ref{factorization property of hardy algebra} below).   Many other algebras like  nest algebras, subdiagonal algebras  etc. and their factorization property have attracted very deep study (see \cite{Ar67, BhMa, Da, La}).

 The next corollary provides a bridge between the theory of $\cst$-extreme maps and factorization property of certain algebras.

\begin{corollary}\label{a factorization property of algebras coming out of C* extreme point}
Let $\phi$ be  a $\cst$-extreme point in $\sha$, and let $(\pi, V,\hpi)$ be its minimal Stinespring triple. If $D$ is any positive and invertible operator in $\copa$, then there exists an invertible operator  $S\in\copa$  such that $D=S^*S$, $SVV^*=VV^*SVV^*$, and $V^*SV$ is invertible with  inverse $V^*S^{-1}V$. In particular, the algebra \begin{equation}\label{eq:algebra with factorization}
    \M=\{T\in\copa; TVV^*=VV^*TVV^*\}
\end{equation}
has factorization in $\copa$.
\end{corollary}
\begin{proof}
Let $D$ be a positive and invertible operator in $\copa$.
Clearly $V^*DV$ is invertible; hence by Theorem \ref{Farenick Zhou criterion} and Observation \ref{observation}, we get a co-isometry $U\in\copa$ with initial space $\overline{\ran(D^{1/2})}$ and an invertible $Z\in\bh$ such that $UD^{1/2}V=VZ$.  Note that $\overline{\ran(D^{1/2})}=\h_\pi$ as $D$ is invertible; so $U$ is unitary. Set $S=UD^{1/2}$. Then $S\in\copa$ and $S$ is invertible. Also  $D=S^*S$ and $SVV^*=VV^*SVV^*$ with $V^*SV$ invertible.
Note that
\begin{equation*}
    (V^*S^{-1}V)(V^*SV)=V^*S^{-1}(VV^*SVV^*)V=V^*S^{-1}(SVV^*)V=V^*(S^{-1}S)V=I_\h,
\end{equation*}
and since $V^*SV$ is invertible, it follows that $(V^*SV)^{-1}=V^*S^{-1}V$.
Further
\begin{equation*}
\begin{split}
   [(I_{\h_\pi}-VV^*)S^{-1}V](V^*SV)&=(I_{\h_\pi}-VV^*)S^{-1}(VV^*SVV^*)V=(I_{\h_\pi}-VV^*)S^{-1}(SVV^*)V\\
   &=(I_{\hpi}-VV^*)(SS^{-1}VV^*V)=(I_{\h_\pi}-VV^*)VV^*V=0.
 \end{split}
\end{equation*}
Since $V^*SV$ is invertible, it follows that $(I_{\h_\pi}-VV^*)S^{-1}V=0$; hence $S^{-1}VV^*=VV^*S^{-1}VV^*$. In particular, $S,S^{-1}\in\M$, so we conclude  that $\M$ has factorization in $\copa$.
\end{proof}

We end this section by considering the question of when  a  $\cst$-extreme point is also extreme, and vice versa. If $\h$ is a finite dimensional Hilbert space, then it was shown in \cite{FaMo} that  every $\cst$-extreme point of $\sha$ is extreme as well. Whether this is true for infinite dimensional Hilbert spaces is not known. Conversely, there are examples where an extreme point in $\sha$ is not $\cst$-extreme (see pg. 1470 in \cite{FaZh}). We discuss some sufficient criteria under which  condition of $\cst$-extremity automatically implies extremity. Also see Corollary \ref{direct sum cst extreme is also extreme} below.

\begin{proposition}\label{multiplicity free}
Let $\phi\in\sha$  with minimal Stinespring triple $(\pi,V,\hpi)$ such that $\pi$ is multiplicity-free (i.e. $\copa$ is commutative). If $\phi$ is $\cst$-extreme in $\sha$, then $\phi$ is extreme in $\sha$.
\end{proposition}
\begin{proof}
To show $\phi$ is extreme in $\sha$, we use Arveson's extreme point condition
(Theorem \ref{Extreme point condition}). Let $D$ be a self-adjoint
operator  in $\copa$ such that $V^*DV=0$. By multiplying by a small
enough scalar,  we assume without loss of generality that
$-\frac{1}{2}I_{\hpi}\leq D\leq\frac{1}{2}I_{\hpi}$. Then
$D+I_{\hpi}$ is positive and invertible. By Corollary \ref{a
factorization property of algebras coming out of C* extreme point},
there exists an invertible $S\in\copa$ satisfying $SVV^*=VV^*SVV^*$
with  $V^*SV$  invertible such that  $D+I_{\hpi}=S^*S$.
Thus we have
\begin{equation*}
    (V^*SV)^*(V^*SV)=V^*S^*(VV^*SVV^*)V=V^*S^*(SVV^*)V=V^*S^*SV=V^*DV+V^*V=I_\h,
\end{equation*}
and since $V^*SV$ is invertible, it follows that $V^*SV$ is  unitary, that is, $V^*SVV^*S^*V=I_\h.$
Further as $\copa$ is commutative by hypothesis, we have $SS^*=S^*S=D+I_{\hpi}$; hence $V^*SS^*V=V^*(D+I_{\hpi})V=I_\h$. Therefore we get
\begin{equation*}
    [V^*S(I_{\hpi}-VV^*)][V^*S(I_{\hpi}-VV^*)]^*= V^*S(I_{\hpi}-VV^*)S^*V=V^*SS^*V-V^*SVV^*S^*V=0.
\end{equation*}
This implies $V^*S(I_{\hpi}-VV^*)=0$, which further yields
\begin{equation*}
      VV^*S=VV^*SVV^*=SVV^*.
\end{equation*}
In other words, $S$ commutes with $VV^*$ which also implies that $S^*$ commutes with $VV^*$; hence $D$ commutes with $VV^*$. Therefore, we have $DV=DVV^*V=VV^*DV=0.$
But then $D\pi(\A)V=\pi(\A)DV=0$ and since $\pi(\A)V\h$ is dense in $\hpi$, we conclude that $D=0$. Since $D$ is arbitrary, this proves that $\phi$ is extreme in $\sha$.
\end{proof}

\section{Direct sums of pure UCP maps}\label{section:direct sum of pure UCP maps}

The  question of whether the direct sum of two $\cst$-extreme  points is also $\cst$-extreme  is very natural. For the case when the Hilbert space  is finite dimensional, a necessary and sufficient criterion for the validity of the assertion is known due to Farenick-Zhou \cite{FaZh}. In fact if $\A$ is a unital $\cst$-algebra and $n\in\N$, then every $\cst$-extreme point in $S_{\C^n}(\A)$ is a direct sum of pure UCP maps (Theorem 2.1, \cite{FaMo}), so the question reduces to finding conditions under which direct sums of pure UCP maps are $\cst$-extreme (which was exploited in \cite{FaZh}). But it is no longer the case in infinite dimensional Hilbert space settings that a $\cst$-extreme point is a direct sum of pure UCP maps (see Example 2, \cite{FaMo}).
Nevertheless, finding  criteria for a direct sum of pure UCP maps in $\sha$ (for $\h$ infinite dimensional) to be $\cst$-extreme is interesting in its own right. In this section, we provide a complete characterization  for such maps to be $\cst$-extreme. One of the main applications of this description would be in proving Krein-Milman type theorem in Section \ref{section:Krein-Milman theorem}.

We begin with some general properties of $\cst$-extremity under direct sums. In the rest of the article, $\Lambda$ will usually be a countable indexing set for a family of maps or subspaces. For any family $\{\phi_i:\A\to\B(\h_i)\}_{i\in\Lambda}$ of UCP maps, their {\em direct sum} $\oplus_{i\in\Lambda}\phi_i$
is the UCP map from $\A$ to $\B(\oplus_{i\in\Lambda}\h_i)$ defined by $(\oplus_{i\in\Lambda}\phi_i)(a)=\oplus_{i\in\Lambda}\phi_i(a)$, for all $a\in\A$. The following remark records the minimal Stinespring triple for a direct sum of UCP maps, which is  easy to verify.

\begin{remark}\label{minimal Stinespring of direct sum}
Let $\phi_i:\A\to\mathcal{B}(\h_i)$, $i\in\Lambda$, be a collection of UCP maps with respective minimal Stinespring triple $(\pi_i,V_i, \K_i)$.
Then  the minimal Stinespring triple for $\oplus_{i\in\Lambda}\phi_i$ is given by $(\pi,V,\K)$, 
where $\K=\oplus_{i\in\Lambda}\K_i$, $V=\oplus_{i\in\Lambda}V_i$ and $\pi=\oplus_{i\in\Lambda}\pi_i$.
\end{remark}

We now recall some  notions relevant to our results.
If $\pi:\A\to\B(\hpi)$ is a representation, and $\K\subseteq\hpi$ is a subspace invariant (and hence reducing) under $\pi(a)$ for all $a\in\A$, then the map $a\mapsto\pi(a)_{|_{\K}}$ is a representation from $\A$ to $\bk$, called {\em sub-representation} of $\pi$.
Two representations $\pi_i:\A\to\B(\h_{\pi_i})$, $i=1,2$, are said to be {\em disjoint} if no non-zero sub-representation of $\pi_1$ is unitarily equivalent to any sub-representation of $\pi_2$.  We shall use the following fact about disjoint representations (see Proposition 2.1.4, \cite{Arbook}): If $\pi_1$ and $\pi_2$ are disjoint representations  and $S\pi_1(a)=\pi_2(a)S$, for all $a\in\A$, for some $S\in\B(\h_{\pi_1},\h_{\pi_2})$, then $S=0$.

A representation $\pi$ is called {\em irreducible} if it has no non-zero sub-representation (equivalently, $\pi(\A)'=\C\cdot I_{\hpi}$). Note that if $\pi_1$ and $\pi_2$ are two non-unitarily equivalent irreducible representations, then $\pi_1(\cdot)\otimes I_{\K_1}$ and $\pi_2(\cdot)\otimes I_{\K_2}$ are disjoint representations (for any Hilbert spaces $\K_1$ and $\K_2$).

We recollect some more terminologies from CP map theory. A completely positive map $\phi$ is called {\em pure} if whenever $\psi$ is a completely positive map with $\psi\leq \phi$, then $\psi=\lambda\phi$ for some $\lambda\in[0,1]$. It is easy to verify that if $(\pi,V,\hpi)$ is the minimal Stinespring triple of a completely positive map $\phi$, then $\phi$ is pure if and only if $\pi$ is irreducible
 (Corollary 1.4.3, \cite{Ar69}). All pure UCP maps are known to be $\cst$-extreme as well as extreme points of $\sha$ (Proposition 1.2, \cite{FaMo}).

Let $\phi_i:\A\to\mathcal{B}(\h_i)$, $i=1,2$, be two UCP maps. We say $\phi_2$ is a \textit{compression} of $\phi_1$ if there exists an isometry $W:\h_2\to\h_1$ such that $\phi_2(a)=W^*\phi_1(a)W$, for all $a\in\A$.
If $\phi$ is a pure UCP map  with the minimal Stinespring triple $(\pi,V,\hpi)$, and $\psi=W^*\phi(\cdot)W$ is a compression of $\phi$ for some isometry $W$, then  $(\pi,VW,\hpi)$ is the minimal Stinespring triple for $\psi$, and so $\psi$ is pure. This follows from the fact that $\pi(\A)'=\C\cdot I_{\hpi}$, so that $\pi(\A)''=\B(\hpi)$, which further yields \[[\pi(\A)VW\h]=[\pi(\A)''VW\h]=[\B(\hpi)VW\h]=\hpi.\]
Moreover, if $(\pi, V_i,\hpi)$ is the minimal Stinespring triple of  UCP maps $\phi_i$, $i=1,2$ (i.e. both $\phi_1,\phi_2$ are compression of the same representation $\pi$), then one can easily show that $\phi_2$ is a compression of $\phi_1$ if and only if $V_2V_2^*\leq V_1V_1^*$ i.e. $\ran(V_2)\subseteq\ran(V_1)$.

Inspired from the notion of disjointness of representations, we define the same for UCP maps as follows. One can see this notion being considered for pure maps in \cite{FaMo}.

\begin{definition}
For any two UCP maps $\phi_i:\A\to\B(\h_i)$, $i=1,2$ with respective minimal Stinespring triple $(\pi_i,V_i,\h_{\pi_i})$, we say $\phi_1$ is {\em disjoint} to $\phi_2$ if $\pi_1$ and  $\pi_2$ are disjoint representations.
\end{definition}

The major results of this paper deal with finding  conditions under which direct sums of mutually disjoint UCP maps (especially, pure maps) are $\cst$-extreme.
The next lemma and proposition are the first step in this direction.

It should be remarked that for  a  family of Hilbert spaces $\{\H_i\}_{i\in\Lambda}$, an operator $T$ in $\B(\oplus_{i\in\Lambda}\h_i)$ will also be written in the matrix form $[T_{ij}]$, for some $T_{ij}\in\B(\h_j,\h_i)$.

\begin{lemma}\label{commutant of direct sum of disjoint representations}
Let $\pi_i:\A\to\B(\K_i)$, $i\in\Lambda$, be a collection of mutually disjoint representations. If $\pi=\oplus_{i\in\Lambda}\pi_i$,
then $\pi(\A)'=\{\oplus_{i\in\Lambda}T_i;~ T_i\in\pi_i(\A)'\}$.
\end{lemma}
\begin{proof}
Let $S\in\pi(\A)'\subseteq\B(\oplus_{i\in\Lambda}\K_i)$. Then $S=[S_{ij}]$ for some $S_{ij}\in\B(\K_j,\K_i)$, such that for all $a\in\A$, we have $[S_{ij}](\oplus_{i\in\Lambda}\pi_i(a))=(\oplus_{i\in\Lambda}\pi_i(a))[S_{ij}]$; hence $S_{ij}\pi_j(a)=\pi_i(a)S_{ij}$ for all $i,j$. For $i\neq j$, since $\pi_i$ is disjoint to $\pi_j$, it follows (from above mentioned result) that $S_{ij}=0$. Also for each $i$,  $S_{ii}\pi_i(a)=\pi_i(a)S_{ii}$ for $a\in\A$, implies that $S_{ii}\in\pi_i(\A)'$. Thus $S=\oplus_{i\in\Lambda}S_{ii}$, where  $S_{ii}\in\pi_i(\A)'$. This shows that $\pi(\A)'\subseteq\{\oplus_{i\in\Lambda}T_i; T_i\in\pi_i(\A)'\}$. The other inclusion is obvious.
\end{proof}

\begin{proposition}\label{direct sum of disjoint C*-extreme maps}
Let $\{\phi_i:\A\to\B(\h_i)\}_{i\in\Lambda}$ be a collection of mutually disjoint UCP maps.
Then $\phi=\oplus_{i\in\Lambda}\phi_i$ is $\cst$-extreme (resp.  extreme) in $ S_{\oplus_{i\in\Lambda}\h_i}(\A)$ if and only if each $\phi_i$ is $\cst$-extreme (resp.  extreme) in $ S_{\h_i}(\A)$.
\end{proposition}
\begin{proof}
Let $(\pi_i,V_i,\K_i)$ be the minimal Stinespring triple for $\phi_i, i\in \Lambda$. Then as noted in Remark \ref{minimal Stinespring of direct sum}, $(\pi,\K, V)$ is the minimal Stinespring triple for $\phi$, where $\K=\oplus_{i\in\Lambda}\K_i$,  $\pi=\oplus_{i\in\Lambda}\pi_i$, and $V=\oplus_{i\in\Lambda}V_i$. Since $\pi_i$ is disjoint to $\pi_j$ for $i\neq j$, it follows from  Lemma \ref{commutant of direct sum of disjoint representations} that
\begin{equation}\label{eq:expression for copa}
\pi(\A)'=\{\oplus_{i\in\Lambda}T_i; T_i\in\pi_i(\A)'\}\subseteq \B(\oplus_{i\in\Lambda}\K_i).
\end{equation}
To prove the equivalent criteria for $\cst$-extremity, we shall use Corollary \ref{a C*-extreme criterian in language of factorization of positive operators}. Assume first that each $\phi_i$ is $C^*$-extreme in $S_{\h_i}(\A)$. Let $D\in\pi(\A)'$ be positive such that $V^*DV$ is invertible. Then it follows from \eqref{eq:expression for copa} that $D=\oplus_{i\in\Lambda}D_i$ for some $D_i\in\pi_i(\A)'$, and hence  $V^*DV=\oplus_{i\in\Lambda}V_i^*D_iV_i$. Clearly each $D_i$ is positive such that $V_i^*D_iV_i$ is invertible satisfying $\sup_{i\in\Lambda}\|(V^*_iD_iV_i)^{-1}\|=\|(V^*DV)^{-1}\|.$ Since each $\phi_i$ is $\cst$-extreme, there exists an operator $S_i\in\pi_i(\A)'$ such that $D_i=S_i^*S_i$, $S_iV_iV_i^*=V_iV_i^*S_iV_iV_i^*$ and $V_i^*S_iV_i$ is invertible. Set $S=\oplus_{i\in\Lambda}S_i$. It is then immediate that $S\in\pi(\A)'$, $D=S^*S$ and $SVV^*=VV^*SVV^*$. Also from Observation \ref{observation}, it follows that
\[\sup_{i\in\Lambda}\|(V_i^*S_iV_i)^{-1}\|^2=\sup_{i\in\Lambda}\|(V_iD_iV_i)^{-1}\|=\|(V^*DV)^{-1}\|<\infty,\] which implies that $V^*SV=\oplus_{i\in\Lambda}V_i^*S_iV_i$ is invertible. Since $D$ is arbitrary, it follows that $\oplus_{i\in\Lambda}\phi_i$ is $\cst$-extreme.

Conversely, let $\oplus_{i\in\Lambda}\phi_i$ be $\cst$-extreme.  Fix
$j\in\Lambda$, and let $D_j\in\pi_j(\A)'$ be a positive operator
such that $V_j^*D_jV_j$ is invertible. For $i\neq j$, let $D_i=I_{\K_i}$
and set $D=\oplus_{i\in\Lambda} D_i$. It is clear that $D\in\pi(\A)'$. Also $D$
is positive and $V^*DV$ is invertible, as each $V_i^*D_iV_i$ is
invertible whose inverse is uniformly bounded. Since $\oplus_{i\in\Lambda}\phi_i$ is
$\cst$-extreme, there is an operator $S\in\pi(\A)'$ such that
$D=S^*S$, $SVV^*=VV^*SVV^*$ and $V^*SV$ is invertible. Again from
\eqref{eq:expression for copa}, we have $S=\oplus_{i\in\Lambda}S_i$ for some
$S_i\in\pi_i(\A)'$. Then the expressions $D=S^*S$ and
$SVV^*=VV^*SVV^*$ imply respectively that $D_j=S_j^*S_j$ and
$S_jV_jV_j^*=V_jV_j^*S_jV_jV_j^*$. Also invertibility of $V^*SV$
implies that $V_j^*S_jV_j$ is invertible. Since $D_j$ is arbitrary,
we conclude that $\phi_j$ is $\cst$-extreme in $S_{\h_j}(\A)$. The
case of equivalence of extreme points can  be proved in a similar
fashion using Arveson's  extreme point criterion (Theorem \ref{Extreme point condition}).
\end{proof}

Some of the subsequent results about direct sums of pure maps involve a strong  connection of their $\cst$-extremity conditions  with the theory of nests of subspaces and corresponding nest algebras. To this end, we recall the basics of nest algebra theory. A family $\E$ of subspaces of a Hilbert space $\h$ is called a {\em nest} if $\E$ is totally ordered by inclusion (i.e. $E\subseteq F$ or $F\subseteq E$ for any $E,F\in\E$). The nest $\E$ is {\em complete} if $0,\h\in\E$, and  $\bigwedge_{F\in\mathcal{F}}F\in\E$ and $\bigvee_{F\in\mathcal{F}}F\in\E$ for any subnest $\mathcal{F}$ of $\E$. Note that any nest $\E$ can be extended   by adjoining $\{0\}, \h$, and  $\wedge$ and $\vee$ of arbitrary subfamily to get the smallest complete nest containing $\E$, which we call the {\em completion} of $\E$ (see Lemma 2.2, \cite{Ri}).  For any nest $\E$, let $\Alg\E$ denote the collection of all operators in $\bh$ which leave subspaces of $\E$ invariant i.e.
\[\Alg\E=\{T\in\bh; T(E)\subseteq E~ \mbox{ for all }E\in\E\}.\]
Clearly $\Alg\E$ is a unital closed algebra, called the {\em nest
algebra} associated with $\E$. If $\overline{\E}$ denotes the
completion of a nest $\E$, then one can readily verify  that
$\Alg\overline{\E}=\Alg\E$. The following theorem  of Larson
\cite{La} about factorization property of certain nest algebras is
very crucial to our results on $\cst$-extreme points. Recall
Definition \ref{definition of algebras with factorization} for
algebras having factorization.

\begin{theorem}[Theorem 4.7, \cite{La}]\label{Larson's factorization result}
Let $\E$ be a  nest in a separable Hilbert space $\h$. Then  $\Alg\E$ has factorization in $\bh$  if and only if the completion $\overline{\E}$ of the nest $\E$ is countable.
\end{theorem}

We are now ready to prove  the main result of this section, which  gives necessary and sufficient criteria for direct sums of UCP maps to be $\cst$-extreme. This generalizes a result of Farenick and Zhou (Theorem 2.1, \cite{FaZh}) from finite to infinite dimensional Hilbert spaces.

Note that if $\phi,\psi:\A\to\bh$ are two pure UCP maps, then either $\phi$ and $\psi$ are mutually disjoint, or they are compression of the same irreducible representation. Therefore in view of Proposition  \ref{direct sum of disjoint C*-extreme maps}, in order  to give criteria of $\cst$-extremity of direct sums of pure UCP maps, it suffices to consider direct sum of only those pure UCP maps which are compression of the same irreducible representation (i.e. those pure maps which are not mutually disjoint), as done in the following theorem.

\begin{theorem}\label{direct sum of pure maps, compression of same representation}
Let $\psi_i:\A\to\B(\h_i)$, $i\in\Lambda$, be a countable family of
non-unitarily equivalent  pure UCP maps with respective minimal
Stinespring triple $(\pi,V_i,\hpi)$, where $\pi $ is a fixed
representation of $\A $, and let $\phi_i=\psi_i(\cdot)\otimes
I_{\K_i}$ for some Hilbert space $\K_i$. Set
$\h=\oplus_{i\in\Lambda}(\h_i\otimes\K_i)$, and
$\phi=\oplus_{i\in\Lambda}\phi_i\in\sha$. Then $\phi$ is
$\cst$-extreme in $\sha$ if and only if the following holds:
\begin{enumerate}
    \item the family $\{\ran (V_i)\}_{i\in\Lambda}$ of subspaces forms a nest in $\hpi$, which induces an order on $\Lambda$ and
    \item if $\L_i=\oplus_{j\leq i}\K_j$ for $i\in\Lambda$, then completion of the nest $\{\L_i\}_{i\in\Lambda}$ in $\oplus_{i\in\Lambda}\K_i$ is countable.
\end{enumerate}
\end{theorem}
\begin{proof}
We know that each $\psi_i$ is unitarily equivalent to the UCP map $a\mapsto P_{\ran(V_i)}\pi(a)_{|_{\ran(V_i)}}$, $a\in\A$. So the fact from the hypothesis that $\psi_i$ and $\psi_j$ are not unitarily equivalent for $i\neq j$ then implies  that $\ran(V_i)\neq \ran(V_j)$, that is,
 \begin{equation}\label{V_iV_i* neq V_jV_j*}
  V_iV_i^*\neq V_jV_j^*,~\mbox{ for all } i\neq j.
\end{equation}
Now set $\h_{\rho}=\oplus_{i\in\Lambda}(\hpi\otimes\K_i)$, and  consider the representation  $\rho:\A\to\B(\h_\rho)$ defined by
\begin{equation*}
    \rho(a)=\oplus_{i\in\Lambda} (\pi(a)\otimes I_{\K_i})~~\mbox{for all }a\in\A,
\end{equation*}
and the isometry $V\in \B(\h,\h_\rho)$ given by
\begin{equation*}
    V=\oplus_{i\in\Lambda}(V_i\otimes I_{\K_i}).
\end{equation*}
It is clear that $(\rho, V, \h_\rho)$  is the minimal Stinespring
triple for $\phi$. We identify the Hilbert space
$\h_\rho=\oplus_{i\in\Lambda}(\hpi\otimes\K_i)$ with the Hilbert
space $\hpi\otimes(\oplus_{i\in\Lambda}\K_i)$; so the representation
$\rho$ is given by
$\rho(a)=\pi(a)\otimes(\oplus_{i\in\Lambda}I_{\K_i})=\pi(a)\otimes
I_{\oplus_{i\in\Lambda}\K_i}$. Since $\pi$ is irreducible,
$\copa=\C\cdot I_{\hpi}$; hence if we consider the operators on the
Hilbert space $\K=\oplus_{i\in\Lambda}\K_i$ in matrix form, then
$\rho(\A)'$ is given by
\begin{equation*}
\begin{split}
\rho(\A)'= (\pi(\A)\otimes I_{\K})'= I_{\hpi}\otimes \bk=
\left\{ I_{\hpi}\otimes [T_{ij}];\; T_{ij}\in\mathcal{B}(\K_j,\K_i)\right\}\subseteq\mathcal{B}\left(\hpi\otimes(\oplus_{i\in\Lambda}\K_i\right)).
\end{split}
\end{equation*}
$\Longrightarrow$ Assume now that $\oplus_{i\in\Lambda}\phi_i$ is a $\cst$-extreme point in $ S_\h(\A)$.
First we show that $\{\ran(V_i)\}_{i\in\Lambda}$ is a nest in $\hpi$. Consider the subalgebra $\M$ of $\B(\oplus_{i\in\Lambda}\K_i)$ given by
\begin{equation}
\begin{split}
    \mathcal{M}&=\left\{[T_{ij}]\in\B(\oplus_{i\in\Lambda}\K_i); \;(I_{\hpi}\otimes [T_{ij}])VV^*=VV^*(I_{\hpi}\otimes [T_{ij}])VV^*\right\}\\
     &=\{[T_{ij}]\in\B(\oplus_{i\in\Lambda}\K_i) ;\;V_jV_j^*\otimes T_{ij}=V_iV_i^*V_jV_j^*\otimes T_{ij}~\forall~ i,j\in\Lambda\}.
\end{split}
\end{equation}
Since $\oplus_{i\in\Lambda}\phi_i$ is $\cst$-extreme, it follows from Corollary \ref{a factorization property of algebras coming out of C* extreme point} that $I_{\hpi}\otimes\M$ has factorization in $\rho(\A)'=I_{\hpi}\otimes \bk$, which is to say that $\M$ has factorization in $\bk$.

Note that if there is an operator $  [T_{ij}]\in\mathcal{M}$ such that  $T_{mn}\neq0$ for some $m,n\in\Lambda$,  then since $V_nV_n^*\otimes T_{mn}=V_mV_m^*V_nV_n^*\otimes T_{mn}$, it will follow that $V_nV_n^*=V_mV_m^*V_nV_n^*$, which further implies $V_mV_m^*\geq V_nV_n^*.$ In other words, we have the following:
\begin{equation}\label{if T_mn is non zero}
    \text{If } V_mV_m^*\ngeq V_nV_n^* \text{ for some } m,n\in\Lambda, \text{ then } T_{mn}=0 \text{ for all }  [T_{ij}]\in\M .
\end{equation}
For the remainder of this implication, we fix $m,n\in\Lambda$ with $m\neq n$. We shall prove that $V_mV_m^*\geq V_nV_n^*$ or $V_nV_n^*\geq V_mV_m^*$. Assume to  the contrary that this is not the case.
 Then it follows from \eqref{if T_mn is non zero} that
 \begin{equation}\label{T_mn=0=T_nm}
 T_{mn}=0 \mbox{  and  }T_{nm}=0,~ \mbox{ for all } [T_{ij}]\in\M.
 \end{equation}
If $\Lambda$ is a two point set, that is, $\Lambda=\{m,n\}$, then $\K=\K_m\oplus\K_n,$ and
with respect to this decomposition,  \eqref{T_mn=0=T_nm} implies that each element $T$ in $\M$ has the form
 \begin{equation*}
     \begin{bmatrix}
    T_1&0\\
    0&T_2
 \end{bmatrix},~
 \mbox{ for } T_1\in\B(\K_m) \mbox{ and } T_2\in\B(\K_n).
 \end{equation*}
  But  if we choose a positive and invertible operator $D$ in $\bk$ of the form $\begin{bmatrix}
    I_{\K_m}&D_1\\
    D_1^*&I_{\K_n}
 \end{bmatrix}$ with  $D_1\in\B(\K_n,\K_m)$ non-zero, then we cannot find any operator $T$ in $\M$ such that $D=T^*T$. This will contradict the fact that $\M$ has factorization in $\bk$.

Therefore we assume for the rest of the implication  that $\Lambda\neq\{m,n\}.$ Now consider the sets
\begin{equation}\label{Lambda_1 definition}
    \Lambda_1=\{l\in\Lambda\setminus\{m,n\};\; T_{lm}=0 \text{ and }T_{ln}=0\;\mbox{ for all }\; [T_{ij}]\in\M \},
\end{equation}
and
\begin{equation*}
    \Lambda_2=\Lambda\setminus(\Lambda_1\cup\{m,n\}).
\end{equation*}
Note that
\[\Lambda_1\cap \Lambda_2=\emptyset ~\text{ and }~ \Lambda_1\sqcup\Lambda_2\sqcup\{m,n\}=\Lambda.\]
Consider the following decomposition:
\begin{equation}\label{decomposition for oplus hpi otimes K_i}
\begin{split}
\K=\oplus_{i\in\Lambda}\K_i= \K_m\oplus\K_n\oplus(\oplus_{i\in\Lambda_1}\K_i)\oplus(\oplus_{i\in\Lambda_2}\K_i)
    =\mathcal{Q}_1\oplus\mathcal{Q}_2\oplus\mathcal{Q}_3\oplus\mathcal{Q}_4,
\end{split}
\end{equation}
where
\[\mathcal{Q}_1=\K_m,~ \mathcal{Q}_2=\K_n,~ \mathcal{Q}_3=\oplus_{i\in\Lambda_1}\K_i, ~\mbox{ and }~ \mathcal{Q}_4=\oplus_{i\in\Lambda_2}\K_i.\]
We shall show that $\mathcal{Q}_3\neq \{0\}$ and $\mathcal{Q}_4\neq \{0\}$ (that is, $\Lambda_1$ and $\Lambda_2$ are non-empty), and that with respect to  decomposition in \eqref{decomposition for oplus hpi otimes K_i}, each $T$ in $\M$ has the following form:
\begin{equation}\label{expression for elements in M}
    T= \begin{bmatrix}
        T_1&0&A_1&0\\
        0&T_2&A_2&0\\
        0&0&X_1&X_2\\
        B_1&B_2&X_3&X_4
    \end{bmatrix},
\end{equation}
for appropriate operators $T_1, T_2,..$ etc. For that, we first
claim the following:  If for some $l\neq m, n,$  there exists an
operator $ [S_{ij}]\in\mathcal{M} $  such that $ S_{lm}\neq0 $ or
$S_{ln}\neq 0,$  then
\begin{equation}\label{S_lm, S_ln neq=0 implies T_lm=0=T_ln}
T_{ml}=0 \mbox{ and }  T_{nl}=0, ~~\forall~ [T_{ij}]\in\M .
\end{equation}
To prove the claim in \eqref{S_lm, S_ln neq=0 implies T_lm=0=T_ln}, assume that $S_{lm}\neq0$, and let $ [T_{ij}]$ be an arbitrary operator in $\M$. Then  it follows from \eqref{if T_mn is non zero}  that $V_lV_l^*\geq V_mV_m^*.$ Since $V_lV_l^*\neq V_mV_m^*$ from \eqref{V_iV_i* neq V_jV_j*}, it follows that $V_mV_m^*\ngeq V_lV_l^*$; again from \eqref{if T_mn is non zero}, we get $T_{ml}=0$. Further, we note that $V_nV_n^*\ngeq V_lV_l^*$ (otherwise we would have $V_nV_n^*\geq V_lV_l^*\geq V_mV_m^*$, and so $V_nV_n^*\geq V_mV_m^*$ which is against our assumption). This in turn implies by \eqref{if T_mn is non zero} that $T_{nl}=0$. Similarly or by symmetry, the condition $S_{ln}\neq 0$ will imply the required claim in \eqref{S_lm, S_ln neq=0 implies T_lm=0=T_ln}.

We now show that $\Lambda_1$ is a non-empty set. Assume otherwise that $\Lambda_1=\emptyset$. Then for each $l\in\Lambda\setminus\{m,n\}$, we have $l\notin\Lambda_1$, so there exists $  [S_{ij}]\in\M$
such that either $S_{lm}\neq0$ or $S_{ln}\neq0$. In either case, \eqref{S_lm, S_ln neq=0 implies T_lm=0=T_ln} implies that for all $T= [T_{ij}]\in\M $, we have $T_{ml}=0$ and $T_{nl}=0$; hence the $(m,n)$ entry of the matrix $TT^*$  satisfies
\begin{equation*}
   \sum_{l\in\Lambda}T_{ml}{T^*_{nl}}=T_{mm}T_{nm}^*+T_{mn}T_{nn}^*+\sum_{l\neq m,n}T_{ml}T_{nl}^*=0,
\end{equation*}
as  $T_{mn}=0$ and $T_{nm}^*=0$ from \eqref{T_mn=0=T_nm}.
Thus for any positive and invertible $D=[D_{ij}]\in\bk$ with $D_{mn}\neq 0$, we cannot find $T\in\mathcal{M}$ such that $D=TT^*$. We can always get such positive and invertible operator $D$ (see the operator in \eqref{expression for tilde D} below).
This violates the fact that $\M^*$ and hence $\M$ has factorization in $\bk$.
Thus our claim that $\Lambda_1\neq\emptyset$ is true.

We next show that $\Lambda_2$ is non-empty. Let if possible, $\Lambda_2=\emptyset$. Then for each $l\in\Lambda$ with $l\neq m,n$, it follows that $l\in \Lambda_1$; hence for all $T= [T_{ij}]\in\M $, we have $T_{lm}=0$ and $T_{ln}=0$, so that  $(m,n)$ entry of $T^*T$ satisfies
\begin{equation*}
  \sum_{l\in\Lambda}T^*_{lm}T_{ln}=0,
\end{equation*}
as $T_{mn}=0$ and $T_{nm}^*=0$.
Again for a positive and invertible operator $ D= [D_{ij}]$ in $\bk$ with $D_{mn}\neq 0$,  we can't find any $T\in\mathcal{M}$ such that $D=T^*T$, violating the fact that $\M$ has factorization in $\bk$. This shows our claim that $\Lambda_2\neq \emptyset$.

Further we note that  if $l\in\Lambda_2$, then $l\notin\Lambda_1$, so  $S_{lm}\neq 0$ or $S_{ln}\neq0$
for some $[S_{ij}]\in\M$; hence it follows from \eqref{S_lm, S_ln neq=0 implies T_lm=0=T_ln} that $T_{ml}=0$ and $T_{nl}=0$ for all $ [T_{ij}]\in\M $. Thus we have
\begin{equation}\label{Lambda_2 property}
    \Lambda_2\subseteq\{l\in\Lambda\setminus\{m,n\}; ~ T_{ml}=0~\mbox{ and }~T_{nl}=0 ~\mbox{ for all }~  [T_{ij}]\in\M \}.
\end{equation}
Now let $T= [T_{ij}]\in\M $, then since $T_{lm}=0$ and $T_{ln}=0$ for all $l\in\Lambda_1$, it follows that
\[P_{\mathcal{Q}_3}T_{|_{\mathcal{Q}_1}}=\sum_{l\in\Lambda_1}P_{\K_l}T_{|_{\K_m}}=\sum_{l\in\Lambda_1}T_{lm}=0,~ \mbox{ and }~ P_{\mathcal{Q}_3}T_{|_{\mathcal{Q}_2}}=\sum_{l\in\Lambda_1}P_{\K_l}T_{|_{\K_n}}=\sum_{l\in\Lambda_1}T_{ln}=0.\]
The sum above is in strong operator topology.
Similarly from \eqref{Lambda_2 property}, since $T_{ml}=0$ and $T_{nl}=0$ for all $l\in\Lambda_2$, it follows that $P_{\mathcal{Q}_1}T_{|_{\mathcal{Q}_4}}=0$ and $P_{\mathcal{Q}_2}T_{|_{\mathcal{Q}_4}}=0$.
These observations along with \eqref{T_mn=0=T_nm} prove our claim that every operator  $T\in\mathcal{M}$ has the form as in \eqref{expression for elements in M}.

 Now with respect to the decomposition in \eqref{decomposition for oplus hpi otimes K_i},
consider the operator $D$ in $\bk$ given by
\begin{equation}\label{expression for tilde D}
    D=\begin{bmatrix}
        I_{\mathcal{Q}_1}&D_1&0&0\\
        D_1^*&I_{\mathcal{Q}_2}&0&0\\
        0&0&I_{\mathcal{Q}_3}&0\\
        0&0&0&I_{\mathcal{Q}_4}
    \end{bmatrix}
\end{equation}
where $D_1\in\mathcal{B}(\mathcal{Q}_2,\mathcal{Q}_1)$ satisfies $0<\|D_1\|<1$. It is then clear that $D$ is a positive and invertible operator in $\bk$. Since $\M$ has factorization in $\bk$, there is  an invertible operator $S\in\mathcal{M}$ with $S^{-1}\in\mathcal{M}$ such that $D=S^*S$. Then from \eqref{expression for elements in M}, $S$ and $S^{-1}$
look like
\begin{equation*}
    S=\begin{bmatrix}
        S_1&0&A_1&0\\
        0&S_2&A_2&0\\
        0&0&X_1&X_2\\
        B_1&B_2&X_3&X_4
    \end{bmatrix}~\text{ and }~
    S^{-1}=\begin{bmatrix}
        T_1&0&C_1&0\\
        0&T_2&C_2&0\\
        0&0&Y_1&Y_2\\
        E_1&E_2&Y_3&Y_4
    \end{bmatrix}.
\end{equation*}
Now
\begin{equation*}
\begin{split}
I_\K=SS^{-1} =
     \begin{bmatrix}
        S_1T_1&0&S_1C_1+A_1Y_1&A_1Y_2\\
        0&S_2T_2&S_2C_2+A_2Y_1&A_2Y_2\\
        X_2E_1&X_2E_2&X_1Y_1+X_2Y_3&X_1Y_2+X_2Y_4\\
       B_1T_1+X_4E_1&B_2T_2+X_4E_2&B_1C_1+B_2C_2+X_3Y_1+X_4Y_3&X_3Y_2+X_4Y_4
    \end{bmatrix}.
\end{split}
\end{equation*}
Thus we get $S_1T_1=I_{\mathcal{Q}_1}$ and $S_2T_2=I_{\mathcal{Q}_2}$. Similarly from the expression $S^{-1}S=I_{\K}$, we get $T_1S_1=I_{\mathcal{Q}_1}$ and $T_2S_2=I_{\mathcal{Q}_2}$. This shows that $ T_1$  and $T_2$  are invertible. Further, from $(4,1)$ entry of $SS^{-1}$, we have $B_1T_1+X_4E_1=0$, which yields
\begin{equation*}
    B_1=-X_4E_1T_1^{-1}=X_4F_1,
\end{equation*}
 where $F_1=-E_1T_1^{-1}$. Also, from $(4,2)$ entry of $SS^{-1}$, we have $B_2T_2+X_4E_2=0$, that is,
 \begin{equation*}
 B_2=-X_4E_2T_2^{-1}=X_4F_2,
\end{equation*}
 where  $F_2=-E_2T_2^{-1}$. Next we note that $(1,2)$ entry of $S^*S$ is $B_1^*B_2$,  and $(1,4)$ entry of $S^*S$ is $B_1^*X_4$.
By substituting  $B_1=X_4F_1$ and $B_2=X_4F_2$, and equating the corresponding entries of $D$,  we get
$F_1^*X_4^*X_4F_2=B_1^*B_2=D_1$ and $F_1^*X_4^*X_4=B_1^*X_4=0$. This implies that $D_1=0$, which is a contradiction. This again violates the fact that $\M$ has factorization in $\bk$. Thus we have shown our claim that $V_nV_n^*\geq V_mV_m^*$ or $V_mV_m^*\geq V_nV_n^*$, which is to say that $\ran(V_n)\supseteq\ran(V_m)$ or $\ran(V_m)\supseteq\ran(V_n)$. Since $m,n\in\Lambda$ are arbitrary, we conclude that  $\E=\{\ran(V_i)\}_{i\in\Lambda}$ is a nest.

Now  we define an order on $\Lambda$ by assigning
\begin{equation}\label{order on lambda}
    i\leq j ~~\mbox{if and only if}~~ V_iV_i^*\leq  V_jV_j^*,
\end{equation}
for any $i,j\in\Lambda$. Since  $V_iV_i^*\neq V_jV_j^*$ whenever  $i\neq j$, the order on $\Lambda$ is well-defined. Also $\Lambda $ is a totally-ordered set, as $\{\ran(V_i)\}_{i\in\Lambda}$ forms a nest of subspaces. For each $i\in\Lambda$, consider the subspace $\L_i$ of $\K=\oplus_{i\in\Lambda}\K_i$ given by
\begin{equation}\label{eq:direct sum of K_i}
    \L_i=\bigoplus_{j\leq i}\K_j.
\end{equation}
Then it is clear that the collection  $\L=\{\L_i; i\in\Lambda\}$ forms a nest in $\K$ such that  $\L_i\subsetneq \L_j$ if and only if $i<j$.
We have to show that the completion $\overline{\L}$ of the nest $\L$ is countable. We claim that
\begin{equation}\label{eq:M=AlgL*}
    \M=(\Alg\L)^*.
\end{equation}
Since $\M$ has factorization in $\bk$, it will then follow from the claim and Proposition \ref{M* has factorization} that  $\Alg\L$ has factorization in $\bk$, which further will imply our requirement using Theorem \ref{Larson's factorization result} that $\overline{\L}$ is countable (as $\K$ is separable).

To show the claim in \eqref{eq:M=AlgL*}, we first note that if an operator $S=[S_{ij}]$ in $\bk$ leaves all subspaces $\{\L_i\}$ invariant, then $S_{ij}=0$ for all $i>j$; hence $\Alg\L=\{[S_{ij}]\in\bk, S_{ij}=0\mbox{ for }i>j\}$, that is,
\begin{equation}
    (\Alg\L)^*=\{[S_{ij}]\in\bk; S_{ij}=0\mbox{ for }i<j\}.
\end{equation}
Now let $[S_{ij}]\in\M$. Then $V_jV_j^*\otimes S_{ij}=V_iV_i^*V_jV_j^*\otimes S_{ij}$  for all  $i,j\in\Lambda.$ For $i<j$, since $V_iV_i^*V_jV_j^*=V_iV_i^*$ and $V_iV_i^*\neq V_jV_j^*$, it forces that $S_{ij}=0$. This shows that $[S_{ij}]\in(\Alg\L)^*$. Thus $\M\subseteq(\Alg\L)^*$. Conversely, if $[S_{ij}]\in(\Alg\L)^*$, then $S_{ij}=0$ for $i<j$; hence $V_jV_j^*\otimes S_{ij}=0=V_iV_i^*V_jV_j^*\otimes S_{ij}$ for $i< j$. On the other hand, for $i\geq j$, we have $V_iV_i^*\geq V_jV_j^*$, so that $V_iV_i^*V_jV_j^*\otimes S_{ij}=V_jV_j^*\otimes S_{ij}$. This shows that $V_iV_i^*V_jV_j^*\otimes S_{ij}=V_jV_j^*\otimes S_{ij}$ for  all $i,j\in\Lambda$, which is to say that $[S_{ij}]\in\M$. Thus we have shown our claim that $\M=(\Alg\L)^*$.
\\
\\
$\Longleftarrow$  To prove the converse implication, assume that the collection $\{\ran(V_i)\}_{i\in\Lambda}$ is a nest (hence $\Lambda$ is a totally ordered set) such that  completion $\overline{\L}$ of the nest $\L=\{\L_i; i\in\Lambda\}$ as in \eqref{eq:direct sum of K_i} is countable. Similar to the claim in  \eqref{eq:M=AlgL*}, we note that $\M=(\Alg\L)^*$. Since $\overline{\L}$ is countable, it follows from Theorem \ref{Larson's factorization result} that $\Alg\L$ has factorization in $\bk$, which is to say that $\M$ has factorization in $\bk$.

Now to show that $\oplus_{i\in\Lambda}\phi_i$ is $\cst$-extreme, we use Corollary \ref{a C*-extreme criterian in language of factorization of positive operators}.  Let  $\widetilde{D}=I_{\hpi}\otimes [D_{ij}]$ be a positive operator in $\rho(\A)'$ such that $V^*\widetilde{D}V$ is invertible.
We claim that $[D_{ij}]$ is invertible. Since $V^*\widetilde{D}V$ is invertible, there exists $\beta>0$ such that  $V^*\widetilde{D}V\geq\beta V^*V$. Then we have
\begin{equation*}
 0\leq V^*\widetilde{D}V-\beta V^*V= [V_i^*V_j\otimes D_{ij}]-\beta [V_i^*V_i\otimes \delta_{ij}I_{\K_i}]= [V_i^*V_j\otimes (D_{ij}-\delta_{ij}\beta I_{\K_i})],
 \end{equation*}
where $\delta_{ij}$ denotes the Kronecker delta. In particular, for every finite subset $\Lambda_0\subseteq\Lambda$, we have
 \begin{equation}\label{D_ij-beta>0}
 \left[V_i^*V_j\otimes (D_{ij}-\delta_{ij}\beta I_{\K_i})\right]_{i,j\in \Lambda_0}\geq 0.
\end{equation}
 Now fix a finite subset $\Lambda_0\subseteq\Lambda$, and let $h_{\Lambda_0}\in\cap_{i\in \Lambda_0}\ran(V_i)$ be a unit vector (which exists because the set $\{\ran(V_i)\}_{i\in \Lambda_0}$ is finite and  is totally ordered). Then there exist unit vectors $h_i\in\h_i$ such that $V_ih_i=h_{\Lambda_0}$ for each $i\in \Lambda_0$. So for any vector $k_i\in\K_i$, $i\in \Lambda_0$, it follows from \eqref{D_ij-beta>0} that
\begin{equation*}
\begin{split}
    0&\leq \sum _{i,j\in \Lambda_0}\left\langle(V_i^*V_j\otimes (D_{ij}-\delta_{ij}\beta I_{\K_i})(h_j\otimes k_j),(h_i\otimes k_i)\right\rangle\\
    &=\sum _{i,j\in \Lambda_0}\left\langle(V_i^*V_jh_j,h_i\right\rangle \left\langle(D_{ij}-\delta_{ij}\beta I_{\K_i}) k_j,k_i\right\rangle
    =\sum _{i,j\in \Lambda_0}\left\langle V_jh_j,V_ih_i\right\rangle \left\langle(D_{ij}-\delta_{ij}\beta I_{\K_i}) k_j,k_i\right\rangle\\
    &=\sum _{i,j\in \Lambda_0}\left\langle h_{\Lambda_0},h_{\Lambda_0}\right\rangle \left\langle(D_{ij}-\delta_{ij}\beta I_{\K_i}) k_j,k_i\right\rangle
    =\sum _{i,j\in \Lambda_0} \left\langle(D_{ij}-\delta_{ij}\beta I_{\K_i}) k_j,k_i\right\rangle.
\end{split}
\end{equation*}
Since $k_i\in\K_i$ for $i\in \Lambda_0$, is arbitrary, we conclude that $[(D_{ij}-\delta_{ij}\beta I_{\K_i})]_{i,j\in \Lambda_0}\geq0$. Also since $\Lambda_0$ is an arbitrary finite subset of $\Lambda$, it follows  that $[(D_{ij}-\delta_{ij}\beta I_{\K_i})]\geq0$ in $\bk$; hence $[D_{ij}]\geq \beta I_\K $ proving our claim that $D=[D_{ij}]$ is invertible.

Therefore, as $\M$ has factorization in $\bk$,  there is an invertible operator $S\in\bk$ such that $S,S^{-1}\in\M$ and $D=S^*S$.
 Set $\widetilde{S}=I_{\hpi}\otimes S$. Clearly  $\widetilde{S}\in\rho(\A)'$ and $\widetilde{D}={\widetilde{S}}^*\widetilde{S}$.
Since $S^{-1}\in\M$, it follows that $\widetilde{S}^{-1}VV^*=VV^*\widetilde{S}^{-1}VV^*$; hence we have
\begin{equation*}
(V^*\widetilde{S}V)(V^*\widetilde{S}^{-1}V)=V^*\widetilde{S}(VV^*\widetilde{S}^{-1}VV^*)V=V^*\widetilde{S}(\widetilde{S}^{-1}VV^*)V=V^*(\widetilde{S}\widetilde{S}^{-1})(VV^*V)=V^*V=I_\h.
\end{equation*}
Likewise we  get $(V^*\widetilde{S}^{-1}V)(V^*\widetilde{S}V)=V^*V=I_\h$. This shows that $V^*\widetilde{S}V$ is invertible. Thus for a given $\widetilde{D}\in\rho(\A)'$ with $V^*\widetilde{D}V$ invertible, we have got $\widetilde{S}\in\rho(\A)'$ such that $\widetilde{D}=\widetilde{S}^*\widetilde{S}$, $\widetilde{S}VV^*=VV^*\widetilde{S}VV^*$ and  $V^*\widetilde{S}V$ is invertible.  We now conclude from Corollary \ref{a C*-extreme criterian in language of factorization of positive operators} that $\phi=\oplus_{i\in\Lambda}\phi_i$ is a $\cst$-extreme point in $\sha$.
\end{proof}

Combining  Theorem \ref{direct sum of pure maps, compression of same
representation} and Proposition \ref{direct sum of disjoint
C*-extreme maps} we have the following complete characterization of
those $\cst$-extreme points which are direct sums of pure UCP maps.

\begin{theorem}\label{direct sum of pure UCP maps}
Let $\phi$ be a direct sum of pure UCP maps in $\sha$, so that $\phi$ is unitarily equivalent to $\bigoplus_{\alpha\in\Gamma}\bigoplus_{i\in\Lambda_\alpha}\psi_\alpha^i(\cdot)\otimes I_{\K_\alpha^i}$, where $\psi_\alpha^i$ is a pure UCP map with minimal Stinespring triple $(\pi_\alpha, V_\alpha^i,\h_{\pi_\alpha})$ such that $\psi_\alpha^i$ is non-unitarily equivalent to $\psi_\alpha^j$  for each $i\neq j$ in $\Lambda _\alpha$, $\alpha\in\Gamma$, and $\pi_\alpha$ is disjoint to $\pi_\beta$ for $\alpha\neq \beta$. Then $\phi$ is $\cst$-extreme  in $\sha$ if and only if the following holds for each $\alpha\in\Gamma$:
\begin{enumerate}
    \item $\{\ran(V_\alpha^i)\}_{i\in\Lambda_\alpha}$ is a nest in $\h_{\pi_\alpha}$, which makes $\Lambda_\alpha$ a totally ordered set, and
    \item if $\L_\alpha^i=\oplus_{j\leq i}\K_\alpha^j$  for $i\in\Lambda_\alpha$, then the completion of the nest $\{\L_\alpha^i\}_{i\in\Lambda_\alpha}$ in $\oplus_{i\in\Lambda_\alpha}\K_\alpha^i$ is countable.
\end{enumerate}
\end{theorem}

\begin{remark}  Based on their results for finite dimensions, Farenick and Zhou  in their remarks towards the end of \cite{FaZh}
suggest  that Condition (1) in Theorem \ref{direct sum of pure UCP
maps} is perhaps sufficient, even in infinite dimensions,  for a
direct sum of pure UCP maps to be $\cst$-extreme. Here in this
Theorem we observe that Condition (1) is to be supplemented with
Condition (2), which is a somewhat more delicate restriction and is
a purely infinite dimensional phenomenon. It has no role to play in
finite dimensions (see Example \ref{an example voilating above theorem} below). 
\end{remark}

\begin{example}\label{an example voilating above theorem}
Let $\G$ be a separable Hilbert space and let  $\{\G_q\}_{q\in\Q}$ be a collection  of subspaces of $\G$ indexed by rationals $\Q$ such that $\G_q\subsetneq\G_p$ for $q< p$. Let $\K$ be another Hilbert space with an orthonormal basis $\{e_q\}_{q\in\Q}$ indexed by $\Q$ and let $P_q$ denote the projection onto the one dimensional subspace $\C e_q$. Consider the space $\h=\oplus_{q\in\Q}(\G_q\otimes \C e_q)\subseteq\G\otimes\K$, and define the UCP map $\phi:\bg\to\B(\h)$ by
\[\phi(X)=\bigoplus_{q\in\Q}P_{\G_q}X_{|_{\G_q}}\otimes P_{q}, \quad X\in\bg.\]
Then it is clear that $\phi$ is a direct sum of pure UCP maps. It is immediate to see that Condition (1) of Theorem \ref{direct sum of pure UCP maps}  is satisfied (since $\{\G_q\}_{q\in\Q}$ forms a nest). On the other hand, if $\L_p=\oplus_{q\leq p}\C e_q$, then $\{\L_p\}_{p\in\Q}$ is a nest whose completion is uncountable (indeed, indexed by reals $\mathbb{R}$). Thus Condition (2) of Theorem \ref{direct sum of pure UCP maps} fails to hold (compare this example with Example \ref{example of non C*-extreme point using indexing set of rational}).
\end{example}

Some straightforward corollaries of Theorems \ref{direct sum of pure maps, compression of same representation} and Theorem \ref{direct sum of pure UCP maps} are immediate as given below.

\begin{corollary}
Let $\phi=\oplus_{i\in\Lambda}\phi_i$ be a direct sum of pure UCP maps $\phi_i$. If $\phi$ is $\cst$-extreme, then for each $i,j\in\Lambda$, either $\phi_i$ and $\phi_j$ are disjoint, or one of $\{\phi_i,\phi_j\}$ is a compression of the other.
\end{corollary}

\begin{corollary}
Let $\phi:\A\to\bh$ be a direct sum of pure UCP maps.
Then $\phi\oplus \phi$ is a $\cst$-extreme point in $ S_{\h\oplus\h}(\A)$ if and only if  $\phi$ is  a $\cst$-extreme point in $\sha$.
\end{corollary}

Since a finite nest containing $\{0\},\h$ is always complete, the following corollary is immediate from Theorem \ref{direct sum of pure maps, compression of same representation}. This result along with Proposition \ref{direct sum of disjoint C*-extreme maps} also recover Theorem 2.1 in \cite{FaZh}.

\begin{corollary}
Let $\{\phi_i:\A\to\B(\h_i)\}_{i=1}^n$ be a finite collection of pure UCP maps with respective minimal Stinespring triple $(\pi,V_i,\hpi)$ (so that each $\phi_i$ is compression of the same irreducible representation $\pi$).
Then $\phi=\oplus_{i=1}^n\phi_i$ is $\cst$-extreme in $S_{\oplus_{i=1}^n\h_i}(\A)$
if and only
if the family $\{\ran(V_i)\}_{i=1}^n$ is a nest.
\end{corollary}

If $\Lambda$ is a subset of the set of integers $\mathbb{Z}$, and  if $\E=\{E_n\}_{n\in\Lambda}$ is a nest in a Hilbert space $\K$ with the property that $E_n\subseteq E_{m}$ for $n<m$, then  the completion of $\E$ is given by the nest $\E\cup\{0,\K, \vee_{n\in\Lambda}E_n, \wedge_{n\in\Lambda}E_n\}$, which is already countable. Thus the following corollary is immediate from Theorem \ref{direct sum of pure maps, compression of same representation}.

\begin{corollary}\label{direct sum of pure maps indexed by natural numbers give C*-extreme point}
Let $\Lambda=\mathbb{N}$ or $\mathbb{Z}$ or ${\mathbb Z}_-$, or $\{
1, 2, \ldots , m\}$ for some $m\in {\mathbb N}$, and let
$\phi_n:\A\to\B(\h_n)$ be a pure UCP map for $n\in\Lambda$.
If  $\phi_n$ is a compression of $\phi_{n+1}$ for each $n$ with $n,
n+1 \in \Lambda $, then the direct sum
$\phi=\oplus_{n\in\Lambda}\phi_n$
is a $\cst$-extreme point in $\sha$, where $\h=\oplus_{n\in\Lambda}\h_n$.
\end{corollary}

We end this section by  giving a necessary and sufficient criterion
for a direct sum of pure UCP maps to be extreme.  Note that in view
of Proposition \ref{direct sum of disjoint C*-extreme maps}, it is
enough to consider direct sums of only those pure UCP maps which are
compression of the same irreducible representation.

\begin{proposition}\label{extreme point criteria for direct sum of pure maps}
Let $\phi_i:\A\to\B(\h_i)$, $i\in\Lambda$, be a family of  pure UCP maps  with respective minimal Stinespring triple $(\pi,V_i,\hpi)$.
Then $\phi=\oplus_{i\in\Lambda}\phi_i$ is extreme in $S_{\oplus_{i\in\Lambda}\h_i}(\A)$ if and only if $V_i^*V_j\neq 0$ for all $i,j\in\Lambda$.
\end{proposition}
\begin{proof}
Set $\h=\oplus_{i\in\Lambda}\h_i$. Note that $(\rho, V, \h_\rho)$ is the minimal Stinespring triple for $\phi$, where $\h_\rho=\oplus_{i\in\Lambda}\hpi$, $\rho=\oplus_{i\in\Lambda}\pi$ and $V=\oplus_{i\in\Lambda}V_i$. Since $\pi$ is irreducible, $\pi(\A)'=\C\cdot I_{\hpi}$; so it follows that
\begin{equation*}
    \rho(\A)'=\left\{[\lambda_{ij}I_{\hpi}];~\lambda_{ij}\in\C\right\}\subseteq\B(\oplus_{i\in\Lambda}\h_{\pi}).
\end{equation*}
First assume that $\phi$ is extreme in $S_{\h}(\A)$,  and fix
$m,n\in\Lambda$. Let $\lambda\neq 0$ in $\C$. Consider the operator
$T=[\lambda_{ij}I_{\hpi}]\in\rho(\A)'$, where $\lambda_{mn}=\lambda$
and $\lambda_{ij}=0$ otherwise. Then $T\neq 0$. Since $\phi$ is
extreme, it follow from Arveson's extreme point condition (Theorem
\ref{Extreme point condition}) that $V^*TV\neq0$. But
$V^*TV=[\lambda_{ij}V_i^*V_j]$, and since  $\lambda_{ij}V_i^*V_j=0$
for all $(i,j)\neq (m,n)$, it follows that  $\lambda V_m^*V_n\neq
0$, showing that $V_m^*V_n\neq 0$.

Conversely, let $V_i^*V_j\neq 0$ for all $i,j\in\Lambda$.  Let
$T=[\lambda_{ij}I_{\hpi}]\in\rho(\A)'$, $\lambda_{ij}\in\C$, be such
that $V^*TV=0$. Then for each $i,j\in\Lambda$, we have
$\lambda_{ij}V_i^*V_j=0$, which yields $\lambda_{ij}=0$; hence
$T=0$. Again  by extreme point condition of Arveson, we conclude
that $\phi$ is extreme in $\sha$.
\end{proof}

The following corollary is another condition (along which
Proposition \ref{multiplicity free}) under which a $\cst$-extreme
map is also extreme.

\begin{corollary}\label{direct sum cst extreme is also extreme}
Let $\phi\in\sha$ decompose as a direct sum of pure UCP maps. If $\phi$ is a $C^*$-extreme point in $\sha$, then $\phi$ is also an extreme point in $\sha$.
\end{corollary}
\begin{proof}
Let $\phi=\oplus_{i\in\Lambda}\phi_i$ for some pure UCP maps $\phi_i, i\in\Lambda$.
By separating out disjoint UCP maps and then invoking Proposition \ref{direct sum of disjoint C*-extreme maps} if needed, we  assume without loss of generality that each $\phi_i$ is a compression of the same irreducible representation, say $\pi$.
Let $(\pi, V_i,\hpi)$ be the  minimal Stinespring triple for $\phi_i$. Since $\phi$ is $\cst$-extreme, it follows from Theorem \ref{direct sum of pure maps, compression of same representation} that either $V_iV_i^*\geq V_jV_j^*$ or $V_jV_j^*\geq V_iV_i^*$ for all $i,j\in\Lambda$. In either case, it is immediate that $V_i^*V_j\neq 0$ for $i,j\in\Lambda$. The required assertion now follows from Proposition \ref{extreme point criteria for direct sum of pure maps}.
\end{proof}

\section{Normal $C^*$-extreme  maps}\label{section:normal UCP maps}

Our attention now shifts towards the study of structure of normal $\cst$-extreme  maps on von Neumann algebras, specifically on type $I$ factors (i.e. $\bg$ for some Hilbert space $\G$). First, we see some basic properties and examples of such maps.  The main result of this section (Theorem \ref{normal C*-extreme maps and iff criteria}) provides necessary and sufficient criteria for normal $\cst$-extreme UCP maps to be direct sums of normal pure UCP maps.

Let $\B\subseteq\bg$ be a von Neumann algebra. Recall that a positive linear map $\phi:\B\to\bh$ is called {\em normal} if whenever $\{X_i\}$ is a net of increasing (or decreasing) self-adjoint operators converging to $X$ in strong operator topology (SOT), then $\phi(X_i)\to\phi(X)$ in SOT.

Let $NS_\h(\B)$ denote the collection of all normal UCP maps from $\B$ to $\bh$. It is clear that $NS_\h(\B)$ itself is  a $\cst$-convex set. Hence one can define and study $\cst$-extreme points of $NS_\h(\B)$ on the lines of Definition \ref{definition of C*-extreme points}, and look into its structure. However we see below (Proposition \ref{normal C*-extreme points does not depend on the set}) that any normal UCP  map on $\B$ is $\cst$-extreme in $NS_\h(\B)$ if and only if it is $\cst$-extreme in $S_\h(\B)$. Therefore it does not matter whether we explore  $\cst$-extremity conditions in the set $NS_\h(\B)$ or the set $ S_\h(\B)$.

\begin{lemma}\label{CP maps dominated by normal maps are normal}
Let $\phi,\psi:\B\to\bh$ be two completely positive maps such that $\psi\leq \phi$. If $\phi$ is normal, then $\psi$ is normal.
\end{lemma}
\begin{proof}
Let $\{X_i\}$ be a net of decreasing positive elements in $\B$ such that $X_i\downarrow 0$ in SOT. Then $\phi(X_i)\to 0$ in SOT, as $\phi$ is normal. As $\psi$ is positive, we note that $\{\psi(X_i)\}$ is a decreasing net of positive elements; hence  $\psi(X_i)\to Y$ in SOT for some positive operator $Y\in\bh$. But since $\psi(X_i)\leq \phi(X_i)$ for all $i$, it follows by taking limit in SOT that $Y\leq 0$; hence $Y=0$.
\end{proof}

\begin{proposition}\label{normal C*-extreme points does not depend on the set}
A normal UCP map $\phi:\B\to\bh$ is $\cst$-extreme in  $NS_\h(\B)$ if and only if  it is $\cst$-extreme in $ S_\h(\B)$.
\end{proposition}
\begin{proof}
 Since $NS_\h(\B)\subseteq S_\h(\B)$, it is immediate that every normal $\cst$-extreme point of  $ S_\h(\B)$ is also a $\cst$-extreme point of $NS_\h(\B)$. Conversely, let  $\phi$ be a $\cst$-extreme point of $NS_\h(\B)$. Let $\phi=\sum_{i=1}^nT_i^*\phi_i(\cdot) T_i$ be a proper $\cst$-convex combination in $ S_\h(\B)$ for some $\phi_i\in  S_\h(\B)$. Then for each $i$, we have $T_i^*\phi_i(\cdot)T_i\leq \phi(\cdot)$, so it follows from Lemma \ref{CP maps dominated by normal maps are normal} that $T_i^*\phi_i(\cdot)T_i$ is normal; hence $\phi_i$ is normal. Since $\phi$ is $\cst$-extreme in $NS_\h(\B)$, there is a unitary $U_i\in\bh$ such that $\phi_i(\cdot)=U_i^*\phi(\cdot)U_i$, as required to prove that $\phi$ is $\cst$-extreme  in $ S_\h(\B)$.
\end{proof}

For the rest of this section, we assume that all von Neumann algebras are of the form $\bg$ for some separable Hilbert space $\G$.
We now recall the well-known structure of normal representations and the Stinespring dilation of normal UCP maps (see Theorem 1.41, \cite{Pi}).

\begin{theorem}\label{Stinespring dilation for normal maps}
Let $\phi:\bg\to\bh $ be a  normal UCP map. Then there exist a separable Hilbert space $\K$ and an isometry $V:\h\to\mathcal{G}\otimes\K$ such that
\[\phi(X)=V^*(X\otimes I_\K )V ~\mbox{ for all }X\in\B,\]
and satisfies the minimality condition: $\mathcal{G}\otimes\K= \ospan\left\{(X\otimes I_\K )Vh; h\in\h,X\in\B\right\}.$
\end{theorem}

In Theorem \ref{Stinespring dilation for normal maps}, if we
recognize the Hilbert space $\G\otimes\K$ as direct sum of $\dim\K $
copies of $\G$, we get the following structure theorem for normal
UCP maps (see Theorem 2.3, \cite{Davies}).

\begin{corollary}\label{expression for normal map as a sum}
Let $\phi:\bg\to\bh$ be a normal UCP map. Then there exists a finite or countable sequence $\{V_n\}_{n\geq1}$ of operators in $\B(\h,\G)$ such that
\begin{equation}
    \phi(X)=\sum_{n\geq 1}V_n^*XV_n~~\mbox{in SOT,}
\end{equation}
for all $X\in\B$.
\end{corollary}

 Note that the commutator of the set $\{X\otimes I_\K;X\in\bg\}$ in $\mathcal{B}(\mathcal{G}\otimes\K)$ is the  algebra $\{I_\G\otimes T;T\in\bk\}$. So a normal UCP map $\phi:\bg\to\bh$ is  pure  if and only if $\dim\K=1$ i.e. $\phi(X)=V^*XV$ for some isometry $V$ from  $\h$ to $\G$.

The $\cst$-extreme condition (Corollary \ref{a C*-extreme criterian in language of factorization of positive operators}) for normal $\cst$-extreme points of $S_\h(\bg)$
translates as follows:

\begin{theorem}\label{subspacecriterion}
Let $\phi:\bg\to\bh $ be a normal UCP map with minimal Stinespring form $\phi(X)=V^*(X\otimes I_\K )V$, for some Hilbert space $\K$. Then $\phi$ is  $C^*$-extreme  in $S_\h(\bg)$ if and only if for any positive operator $D\in\bk$ with $V^*(I_\G\otimes D)V$ invertible, there exists $S\in\bk$ such that $D=S^*S$, $(I_\G\otimes S)VV^*=VV^*(I_\G\otimes S)VV^*$ and  $V^*(I_\G\otimes S)V$ is invertible.
\end{theorem}

\begin{remark}
Let $\phi:\bg\to\bh $ be a normal UCP map with minimal Stinespring
form $\phi(X)=V^*(X\otimes I_\K )V$. We identify the subspace  $V\h$
with $\h$, so that $\h$ is a subspace of $\G\otimes\K$. It then
follows from Theorem \ref{subspacecriterion}  that $\phi$ is a
$C^*$-extreme  point in $S_\h(\bg)$ if and only if the subspace $\h$
of $\mathcal{G}\otimes\K$ satisfies the following {\em factorization
property}:
\begin{enumerate}[label=(\textdagger)]
    \item\label{Property P}  for any positive operator $D\in\bk$ with $P_\h(I_\G\otimes D)_{|_\h}$   invertible,
     there exists $S\in\bk$ satisfying $D=S^*S$,  $(I_\G\otimes S)(\h)\subseteq\h$ and  $(I_\G\otimes S)_{|_\h}$ is invertible.
\end{enumerate}
 Therefore, in order to understand the structure of  normal $C^*$-extreme maps, one can characterize
   subspaces of $\mathcal{G}\otimes\K$ with factorization property \ref{Property P}.
\end{remark}

The following proposition provides a family of examples of subspaces
in $\G\otimes\K$ satisfying factorization property \ref{Property P}.

\begin{proposition}\label{direct sum of nest satisfies Property P}
Let $\h=\bigvee_{i\in\Lambda}\G_i\otimes \K_i$ be a subspace of $\mathcal{G}\otimes\K$, for some family $\{\G_i\}_{i\in\Lambda}$ and $\{\K_i\}_{i\in\Lambda}$  of subspaces of $\G$ and $\K$ respectively, such that $\G\otimes\K=\ospan\{(X\otimes I_\K)h; h\in\h, X\in\bg\}$. If either of the following is true:
\begin{enumerate}
    \item $\G_i\perp\G_j $ for all $i\neq j $ and
    $\{\K_i\}$ is a  nest whose completion is countable,
    \item  $\{\G_i\}$ is a  nest and $\K_i\perp\K_j$ for $i\neq j$ such that the completion of the nest $\{\oplus_{j\leq i}\K_j\}_{i\in\Lambda}$ is countable,
\end{enumerate}
then $\h$ satisfies factorization property \ref{Property P}.
\end{proposition}
\begin{proof}
(1) Firstly it is easy to verify that $\K=\bigvee_{i\in\Lambda}\K_i$ (indeed, if $k\in\K\ominus\bigvee_{i\in\Lambda}\K_i$, then for any non-zero $g\in\G$, we will have $g\otimes k\perp \{(X\otimes I_\K)h; h\in\h, X\in\bg\}$, which will yield $g\otimes k=0$).

 Let $D\in\bk$ be a positive  operator such that $P_\h(I_\G\otimes D)_{|_{\h}}$ is invertible. We claim that $D$ is invertible.  Let $\beta>0$ be such that $P_\h(I_\G\otimes D)_{|_\h}\geq \beta I_\h$. Since  $g_ i \otimes k_ i \in\h$, for any $0\neq g_ i \in\G_ i $ and $k_ i \in \K_ i $, we get
\begin{equation*}
\begin{split}
    \|g_ i \|^2\langle Dk_ i ,k_ i \rangle&=\langle(I_\G\otimes D)(g_ i \otimes k_ i ),g_ i \otimes k_ i \rangle
    \geq \beta \;\langle g_ i \otimes k_ i ,g_ i \otimes k_ i \rangle
    =\beta\;\|g_ i \|^2\;\langle k_ i ,k_ i  \rangle,
\end{split}
\end{equation*}
which implies that $\langle D k_ i ,k_ i \rangle\geq\beta\langle k_ i ,k_ i \rangle$. Since $\bigcup_{ i \in\Lambda}\K_ i $ is dense in $\K$, we conclude that $\langle Dk,k\rangle\geq\beta\langle k,k\rangle$ for all $k\in\K$; hence $D$ is invertible.

Since the nest $\{\K_i\}_{i\in\Lambda}$  has a countable completion, by Theorem \ref{Larson's factorization result} there exists an invertible operator $S\in\bk$ satisfying $D=S^*S$ and  $S(\K_i)\subseteq\K_i$, $S^{-1}(\K_i)\subseteq\K_i$ for all $i\in\Lambda$. Clearly then $(I_\G\otimes S)(\h)\subseteq\h$.
Note that $(S^{-1})_{|_{\K_ i }}=(S_{|_{\K_ i }})^{-1}\in\mathcal{B}(\K_ i )$ for each $ i \in\Lambda$
and $\sup_{ i \in\Lambda}\|(S_{|_{\K_ i }})^{-1}\|=\|S^{-1}\|<\infty$.
Hence $\oplus_{ i \in\Lambda}I_{\G_ i }\otimes (S_{|_{\K_ i }})^{-1}$ is a bounded operator on $\h$ and
\begin{equation*}
\begin{split}
    (I_\G\otimes S)_{|_\h}(\oplus_{ i \in\Lambda}I_{\G_ i }\otimes (S_{|_{\K_ i }})^{-1})&= (\oplus_{ i \in\Lambda}I_{\G_ i }\otimes S_{|_{\K_ i }})(\oplus_{ i \in\Lambda}I_{\G_ i }\otimes (S_{|_{\K_ i }})^{-1})
    =\oplus_{ i \in\Lambda}I_{\G_ i }\otimes I_{\K_ i }=I_\h.
\end{split}
\end{equation*}
Similarly, $(\oplus_{ i \in\Lambda}I_{\G_ i }\otimes (S_{|_{\K_ i
}})^{-1})(I_\G\otimes S)_{|_\h}=I_\h$.   This proves that
$(I_\G\otimes S)_{|_\h}$ is invertible. Since $D\in\bk$ is
arbitrary, we have shown that $\h$ satisfies factorization property
\ref{Property P}.

(2) This assertion follows from Theorem \ref{direct sum of pure maps, compression of same representation}, as the map $\phi(X)=P_\h(X\otimes I_\K)_{|_\h}=\oplus_{i\in\Lambda}(P_{\G_i} X_{|_{\G_i}}\otimes I_{\K_i})$ from $\bg$ to $\bh$ satisfies the equivalent criteria for it to be $\cst$-extreme in $S_\h(\bg)$.
\end{proof}

At this point, we are not sure if we can write subspaces of Part (1) in Proposition \ref{direct sum of nest satisfies Property P} in the form of subspaces in Part (2), and vice versa.  However one can easily verify that if the concerned nests are already complete, then the two parts produce the same set of subspaces.

Before proving the main result of this section, we recall some  terminologies for the purpose. Let $\E$ be a complete  nest on a separable Hilbert space $\K$. For any $E\in\E$, define
\[E_-=\vee \{F\in\E; \; F\subsetneq E\}~ \mbox{ and }~ E_+=\wedge\{F\in\E;\; E\subsetneq F\}.\]
An {\em atom} of $\E$ is a subspace of the form $E\ominus E_-$, for some $E\in\E$ with $E\neq E_-$.  Clearly, any two atoms of $\E$ are orthogonal.  The nest $\E$ is called {\em atomic} if there is a countable collection of atoms $\{\K_n\}$ of  $\E$ such that $\K=\oplus_n \K_n$.

Now let $\M$ be a subalgebra of $\bk$ for some  Hilbert space $\K$. Then its {\em lattice} $\Lat\M$ is defined by
\[\Lat\M=\{E\subseteq \K;\; E \mbox{ is a subspace such that } T(E)\subseteq E\mbox{ for all } T\in\M\}.\]
Dually, for any collection $\E$ of subspaces of $\K$, consider the unital closed algebra $\Alg\E$ defined by
\[\Alg\E=\{T\in\bk;~ T(E)\subseteq E\;\mbox{ for all }~E\in\E\}.\]
It is clear that $\M\subseteq\Alg\Lat\M$. The subalgebra $\M$  is called {\em reflexive} if $\M=\Alg\Lat\M$. Any nest algebra is an example  of reflexive algebra. More generally, any algebra of the form $\Alg\E$ (for some collection $\E$ of subspaces) is reflexive. One can refer to \cite{Da} for more details.

We now mention a crucial result proved in \cite{BhMa} about reflexive algebras having factorization property to be applied below (see Corollary 2.11 and Lemma 4.3 in \cite{BhMa}): If $\M$ is a reflexive algebra having factorization in $\bk$ for some separable Hilbert space $\K$, then $\Lat\M$ is a complete atomic nest, and hence $\M$ is a nest algebra.

The following theorem is another major result of the article, which provides a necessary and sufficient criteria for a normal $\cst$-extreme map to be direct sum of normal pure UCP maps.

\begin{theorem}\label{normal C*-extreme maps and iff criteria}
 Let $\phi:\bg\to\bh$ be a normal $C^\ast$-extreme map with minimal Stinespring form $\phi(X)=V^*(X\otimes I_{\K})V,$ for some Hilbert space $\K$.
Then  $\phi$ is unitarily equivalent to a direct sum of normal pure UCP maps
 if and only if the algebra $\M=\{T\in\bk; (I_\G\otimes T)(V\h)\subseteq V\h\}$ is reflexive.
 \end{theorem}
\begin{proof}
By identifying the Hilbert space $\h$ with $V\h$, we assume  that
$\h$ is a subspace of $\G\otimes\K$, so that $\phi(X)=P_\h(X\otimes
I_\K)_{|_{\h}}$ for $X\in\bg$ and $\M=\{T\in\bk; (I_\G\otimes
T)\h\subseteq\h\}$.

First we assume that the algebra $\M$ is reflexive. Since $\phi$ is $\cst$-extreme in $S_\h(\bg)$, it follows from Corollary \ref{a factorization property of algebras coming out of C* extreme point} that $I_\G\otimes\M$ has factorization in $I_\G\otimes\bk$, which is to say that $\M$ has factorization in $\bk$. Since $\K$ is separable, it then follows from the result from Corollary 2.11 in \cite{BhMa} as mentioned above,  that $\Lat\M$ is an atomic nest.
Therefore by definition of atomic nests, there exists an orthonormal basis $\{e_n\}_{n\geq 1}$ of $\K$ such that each $e_n$ is contained in  one of the atoms of $\Lat\M$. Now for all $n\geq 1$, consider the subspace $\G_n$ of $\G$ given by
\[\G_n=\{g\in\G; g\otimes e_n\in\h\}.\]
We claim that
\begin{equation}\label{eq:expression for H}
    \h=\bigoplus_{n\geq1}(\G_n\otimes e_n).
\end{equation}
Clearly, $\G_n\otimes e_n\subseteq\h$ for all $n\geq1$; hence
$\oplus_{n\geq1}(\G_n\otimes e_n) \subseteq\h$. Conversely, let
$h\in\h$. Then as $\{e_n\}_{n\geq 1}$ is an orthonormal basis of
$\K $ and $h\in\G\otimes\K$, there exists a sequence $\{g_n\}_{n\geq1}$ of vectors in $\G$
such that
\[h=\sum_{n\geq1}g_n\otimes e_n.\]
Now for any unit vector $e\in\K$,  we denote by  $|e\rangle\langle
e|$  the rank one  projection on $\K$ defined by
 \[|e\rangle\langle e|(k)=e\langle e,k\rangle~~\mbox{ for all }k\in\K.\]
Then we note for all $n\geq1$ that $|e_n\rangle\langle e_n|\in \Alg\Lat\M$  (indeed, if $E\ominus E_{-}$ is an atom of $\Lat\M$ and $e\in E\ominus E_{-}$ is a unit vector, then $|e\rangle\langle e|(F)=0\subseteq F$ for $F\subseteq E_{-}$, and $|e\rangle\langle e|(F)=\C\cdot e\subseteq F$ for $F\supseteq E$). Since $\M$ is reflexive, it then follows  that $|e_n\rangle \langle e_n|\in\M$; hence  $(I_\G\otimes |e_n\rangle \langle e_n|)\h\subseteq\h$, which implies 
\[(I_\G\otimes |e_n\rangle \langle e_n|)h=g_n\otimes e_n\in\h.\]
In particular, $g_n\in\G_n$ and hence $g_n\otimes e_n\in \G_n\otimes
e_n$.  This shows that $h=\sum_{n\geq1}g_n\otimes
e_n\in\oplus_{n\geq1}\G_n\otimes e_n$. Since $h\in\h$ is arbitrary,
we conclude our claim that $\h=\oplus_{n\geq 1}(\G_n\otimes e_n)$.
Now for each $n\geq 1$, define the map $\phi_n:\bg\to \B(\G_n)$ by
\[\phi_n(X)=P_{\G_n}X_{|_{\G_n}}, ~~\mbox{ for all }~X\in\bg.\]
Note that $\G_n$ can be a zero subspace, in which case we ignore the map $\phi_n$.
Then it is clear that $\phi_n$ is a normal pure UCP map, and  for all $X\in\bg$ we have
\[\phi(X)=P_{\h}(X\otimes I_\K)_{|_\h}=\sum_{n\geq1}P_{\G_n}X_{|_{\G_n}}\otimes
 |e_n\rangle\langle e_n|=\bigoplus_{n\geq 1}\phi_n(X)\otimes |e_n\rangle\langle e_n|.\]
This proves the required assertion that $\phi$ is unitarily equivalent to a direct sum of normal pure UCP maps $\phi_n$.

To prove the converse,   let $\phi$ be a direct sum of normal pure
UCP maps. Then for some countable indexing set $J,$ there is a
collection  $\{\G_i \}_{ i \in J}$ of  distinct subspaces of $\G$
and a collection $\{\K_ i \}_{ i \in J}$ of mutually orthogonal
subspaces of $\K$  such that $\h=\oplus_{ i \in J} (\G_ i \otimes\K_
i )$. Since $\phi$ is $\cst$-extreme in $S_\h(\bg)$, the collection
$\{\G_ i \}_{ i \in J}$ is a nest by Theorem \ref{direct sum of pure
maps, compression of same representation}. This nest induces an
order on $ J$ making it a totally ordered set. If we set $\L_ i
=\oplus_{j\geq i }\K_j$ for $i\in J$, then $\{\L_i \}_{ i \in J}$ is
a nest, and it is easy to verify that $\M=\{T\in\bk; (I_\G\otimes
T)(\h)\subseteq\h\}=\Alg\{\L_ i ; i\in J\}$ (to show this, one can
imitate the same argument as in  \eqref{eq:M=AlgL*} in the proof of
Theorem \ref{direct sum of pure maps, compression of same
representation}). Thus we conclude that $\M$ is reflexive.
\end{proof}

It is a known fact due to Juschenko \cite{Ju} that any subalgebra having factorization in the finite dimensional matrix algebra $M_n$ is a nest algebra, and hence is automatically reflexive (also see  \cite{BhMa} for an alternative proof).  Thus the following corollary is immediate from Theorem \ref{normal C*-extreme maps and iff criteria} and Theorem \ref{direct sum of pure maps, compression of same representation}.

\begin{corollary}\label{when K is finite dimensional}
Let $\h$ be a subspace of $\G\otimes \K$, where $\K$ is a finite dimensional Hilbert space, such that the normal UCP map $\phi:\bg\to\bh$ given by $\phi(X)=P_\h(X\otimes I_\K)_{|_{\h}}$, for $X\in\bg$, is in minimal Stinespring form. Then $\phi$ is $\cst$-extreme  in $S_\h(\bg)$ if and only if $\phi$ is unitarily equivalent to a direct sum of a finite sequence of normal pure UCP maps $\{\phi_i\}_{i=1}^n$ such that $\phi_i$ is a compression of $\phi_{i+1}$.
\end{corollary}

Corollary \ref{when K is finite dimensional}  was proved for the
case when the UCP map is from $M_n$ to $M_r$ for some $n,r\in\N$
(Theorem 4.1, \cite{FaMo}) through rather tedious matrix
computations. Here we have provided a more conceptual approach using
nest algebra theory.

This Corollary suggests that  perhaps the algebra $\M$ in Theorem
\ref{normal C*-extreme maps and iff criteria} is always reflexive
when $\phi$ is $\cst$-extreme. But we are not able to prove it. If
this turns out to be true, then Theorem \ref{normal C*-extreme maps
and iff criteria} along with Theorem \ref{direct sum of pure maps,
compression of same representation} would characterize all normal
$\cst$-extreme maps on $\bg$. Thus we propose the following
conjecture:

\begin{conjecture}
Every normal $\cst$-extreme map on a type $I$ factor is a direct sum
of normal pure UCP maps.
\end{conjecture}

\section{A Krein-Milman type theorem}\label{section:Krein-Milman theorem}

The Krein-Milman theorem is a very important result in classical functional analysis, which says that in a locally convex topological vector space, a convex compact subset  is closure of the convex hull of its extreme points. So it is desired to have an analogue of Krein-Milman  theorem for $\cst$-convexity in the space $\sha$ equipped with an appropriate topology.  We equip the set $\sha$  with bounded weak (BW) topology. Convergence in BW-topology is given by: a net $\phi_i$ converges to $\phi$ in $\sha$ if $\phi_i(a)\to \phi(a)$ in weak operator topology (WOT) for all $a\in\A$. It is known that $\sha$ is a compact set with respect to BW-topology. See \cite{Ar69, Pa} for more details on this topology.

So a generalized Krein-Milman theorem for $\sha$ would be to ask whether $\sha$ is the closure of the $\cst$-convex hull of its $\cst$-extreme points in BW-topology.  Here the {\em $\cst$-convex hull} of any subset $K$ of $\sha$ is given by \begin{equation}
    \left\{\sum_{i=1}^nT_i^*\phi_i(\cdot)T_i;\; \phi_i\in K, T_i\in\bh\mbox{ with }\sum_{i=1}^nT_i^*T_i=I_\h\right\}.
\end{equation}
The goal of this section is to prove a Krein-Milman type theorem for  $\sha$,
whenever $\A$ is a separable $\cst$-algebra or $\A$ is of the form
$\bg$ for some Hilbert space $\G$. The proof of these two cases are
different. Note that $\bg$ is not separable, when $\G$ is infinite
dimensional. As yet we do not know the result in full generality
(i.e. for non-separable $\cst$-algebras). We recall here  that as
mentioned before such a theorem can be found in \cite{FaMo} for
$\sha$, when $\h$ is a finite dimensional Hilbert space and $\A$ is
an arbitrary $\cst$-algebra, and in \cite{BaBhMa} for general $\h$
and commutative $C^*$-algebra $\A .$ Thus our result provides an important development towards this theorem in infinite dimensional  Hilbert space settings.

\begin{lemma}\label{sum of pure maps is in C*-convex hull}
Let $\phi\in\sha$ be such  that $\phi(a)=\sum_{n\geq1}\phi_n(a)$ in
WOT, for all $a\in\A$, where $\{\phi_n:\A\to\bh\}_{n\geq1}$ is a
countable family of pure completely positive maps. Then $\phi$ is in
the BW-closure of $\cst$-convex hull of $\cst$-extreme points of
$\sha$.
\end{lemma}
\begin{proof}
We assume that the collection $\{\phi_n\}_{n\geq1}$  in the sum of $\phi$  is countably infinite. The finite case follows similarly and easily. For each $n\geq 1$, let $(\pi_n,V_n,\K_n)$ be the minimal Stinespring triple for $\phi_n$. Then each $\pi_n$ is irreducible, as $\phi_n$ is pure by hypothesis. Note that
\begin{equation*}
    \sum_{n\geq 1}V_n^*V_n=\sum_{n\geq1}\phi_n(1)=\phi(1)=I_\h, ~~\mbox{ in WOT}.
\end{equation*}
Set $A_n=V_n^*V_n\in\bh$, and  let $V_n=W_n A_n^{1/2}$ be the polar decomposition of $V_n$. Here $W_n\in\mathcal{B}(\h,\K_n)$ is the partial isometry with initial space $\overline{\ran(A_n^{1/2})}$ and final space $\overline{\ran(V_n)}$. Define the map $\zeta_n:\A\to\bh$ by \begin{equation*}
    \zeta_n(a)=W_n^*\pi_n(a)W_n~~\mbox{ for all }~a\in\A.
\end{equation*}
It is immediate to verify that $\zeta_n$ is a completely positive map with the minimal Stinespring triple $(\pi_n, W_n, \K_n)$.
Let $\theta_n:\A\to\C$ be a pure state that is a compression of $\zeta_n$ (e.g. take a unit vector $e_n\in\ran(W_n)$ and define $\theta_n(a)=\langle e_n,\pi_n(a)e_n\rangle$ for all $a\in\A$). Now we define $\xi_n:\A\to\bh$ by \begin{equation*}
    \xi_n=\zeta_n+(1-P_n)\theta_n,
\end{equation*}
where $P_n=W_n^*W_n$ is the projection from $\h$ onto $\overline{\ran(A_n^{1/2})}$. Note that $\xi_n$ is a UCP map from $\A$ to $\bh$. If we set $U_n={W_n}_{|_{\ran(P_n)}}$ (so that $U_n$ is an isometry from $\ran(P_n)$ to $\K_n$),
then it is straightforward to verify that $\xi_n$ is unitarily equivalent to  the UCP map $\widetilde{\xi}_n:\A\to \B(\ran(P_n)\oplus\ran(P_n^\perp))$ given by
\[\widetilde{\xi}_n(a)= U_n^*\pi_n(a)U_n\oplus \theta_n(a)I_{\ran(P_n^\perp)}, ~~\mbox{ for all }~a\in\A.\]
Since $\theta_n$ is a compression of the map $a\mapsto U_n^*\pi_n(a)U_n$ (which is pure, as $\pi_n$ is irreducible), it follows
from Theorem \ref{direct sum of pure maps, compression of same representation} that $\widetilde{\xi}_n$ is  $C^*$-extreme in $S_{\ran(P_n)\oplus\ran(P_n^\perp)}(\A)$; hence $\xi_n$ is $\cst$-extreme in $\sha$.

 Now set $B_n=I_\h-\sum_{j=1}^nA_j$. Since $\sum_{n\geq 1}A_n=\sum_{n\geq1}V_n^*V_n=I_\h$ in WOT;
it follows that $B_n\geq0$, and  $B_n\to 0$ in WOT as $n\to\infty$.
Now fix a $C^*$-extreme point $\xi$ in $\sha$ and define the map $\psi_n:\A\to\bh$ by
\begin{equation*}
\psi_n(a)=B_n^{1/2}\xi(a)B_n^{1/2}+\sum_{j=1}^nA_j^{1/2}\xi_j(a)A_j^{1/2},~~\mbox{ for all }~a\in\A.
\end{equation*}
It is clear  that each $\psi_n$  is a UCP map such that $\psi_n$ is a $C^*$-convex combination of $C^*$-extreme points of $\sha$.
Since $B_n\to0$ in WOT, it follows that $B_n^{1/2}\to0$ in SOT; hence $B_n^{1/2}\xi(a)B_n^{1/2}\to 0$ in WOT for all $a\in\A$. This implies that
\[\lim_{n\to\infty}\psi_n(a)=\sum_{j=1}^\infty A_j^{1/2}\xi_j(a)A_j^{1/2}~~\mbox{ in WOT, for all  } a\in\A.\]
Note that $A_j^{1/2}(I-P_j)=0$ for all $j$. Hence for all $a\in\A$, we get $A_j^{1/2}\xi_j(a)A_j^{1/2}=A_j^{1/2}\zeta_j(a)A_j^{1/2}$, which further yields in WOT convergence
\begin{equation*}
    \lim_{n\to\infty}\psi_n(a)=\sum_{j=1}^\infty A_j^{1/2}\zeta_j(a)A_j^{1/2}=\sum_{j=1}^\infty A_j^{1/2}W_j^*\pi_j(a)W_jA_j^{1/2}=\sum_{j=1}^\infty V_j^*\pi_j(a)V_j=\sum_{j=1}^\infty\phi_j(a)=\phi(a).
\end{equation*}
In other words, $\psi_n\to\phi$ in BW-topology. Thus we have approximated $\phi$ in BW-topology by a sequence $\psi_n$ belonging to the $\cst$convex hull of $\cst$-extreme points of $\sha$.
\end{proof}


The following proposition seems to be a well-known result. However, we could trace the proof only when $\h$ is a finite dimensional Hilbert space. So we outline a proof for the sake of completeness.

\begin{proposition}\label{UCP maps are approximated by normal UCP maps}
Let $\A$ be a von Neumann algebra, and let $\phi:\A\to\bh$ be a UCP
map.  Then there exists a sequence $\phi_n:\A\to\bh$ of normal UCP
maps such that $\phi_n(a)\to\phi(a)$ in SOT for all $a\in \A$. In
particular, the set  $NS_\h(\A)$ of normal generalized states is
dense in the set $S_\h(\A)$ of all generalized states in
BW-topology.
\end{proposition}
\begin{proof}
If $\h$ is finite dimensional, then the assertion is proved in (Corollary 1.6.3, \cite{BrOz}). So assume that $\h$ is infinite dimensional.
Let $\{P_n\}_{n\geq1}$ be an increasing sequence of projections on $\h$ with finite dimensional ranges such that $P_n\to I_\h$ in SOT. Fix a normal UCP map $\psi:\A\to\bh$, and for each $n\geq1$, consider the map $\phi_n:\A\to\bh$ given by
\[\phi_n(a)=P_n\phi(a)P_n+(1-P_n)\psi(a)(1-P_n),~~\mbox{ for all } a\in\A.\]
Since $P_n\to I_\h$ in SOT, we note that $\phi_n(a)\to\phi(a)$ in SOT for all $a\in\A$. Also the second term in the
above sum is normal, as $\psi$ is normal. So it suffices to approximate the map
$P_n\phi(\cdot)P_n$ by normal completely positive maps. The problem now reduces to
approximation of (unital) completely positive maps by normal (unital) completely positive maps acting on finite
dimensional Hilbert spaces,  which is possible as already noted.
\end{proof}

We are now ready to show a  Krein-Milman type theorem for $\sha$ for the case when $\A$ is a separable $\cst$-algebra or a type $I$ factor, and $\h$ is an infinite dimensional Hilbert space.

\begin{theorem}\label{Krein-Milman type theorem}
Let $\A$ be a separable $C^*$-algebra or a type $I$ factor, and let $\h$ be a separable Hilbert space. Then
$\sha$ is BW-closure of $C^*$-convex hull of its $C^*$-extreme points.
\end{theorem}
\begin{proof}
\textit{Case I:} First assume that $\A$ is a separable $\cst$-algebra.
Let $\phi\in\sha$, and let $(\pi,V,\hpi)$ be its minimal Stinespring triple. Since both $\A$ and $\h$ are separable, the Hilbert space $\hpi$ is also separable. By a corollary of  Voiculescu's theorem (see
Theorem 42.1, \cite{Co}), there exists a sequence $\{U_n\}$ of unitaries on $\hpi$ and a representation $\rho:\A\to\B(\hpi)$ such that $\rho$ is a direct sum of irreducible representations and \begin{equation*}
    \pi(a)=\lim_{n\to\infty}U_n^*\rho(a)U_n~~\mbox{ in WOT},
\end{equation*}
for all $a\in\A$. Therefore if we set $W_n=U_nV$, then each $W_n$ is an isometry, and $\phi(a)=\lim_{n\to\infty}W_n^*\rho(a)W_n$ in WOT for all $a\in\A$. In other words, $\phi$ is approximated in BW-topology by UCP  maps, all of which are compression of the representation $\rho$ that is a direct sum of irreducible representations. Thus without loss of generality, we  assume that $\pi$  itself is a direct sum of a finite or countable irreducible representations, say,
\begin{equation}
    \pi=\oplus_{n\geq1}\pi_n,
\end{equation}
where  $\pi_n:\A\to\mathcal{B}(\K_n)$ is an irreducible representation on some Hilbert space $\K_n$. Now for each $n\geq1$, let $Q_n$ denote the projection of $\hpi$ onto $\K_n$, and let $V_n=Q_nV\in\mathcal{B}(\h,\K_n)$. Consider the completely positive map $\phi_n:\A\to\bh$  defined by $\phi_n(a)=V_n^*\pi_n(a)V_n$ for all $a\in\A$. Since $\pi_n$ is irreducible, each $\phi_n$ is a pure completely positive map. Also note that in WOT convergence, we have
\begin{equation*}
    \sum_{n\geq1}\phi_n(a)=\sum_{n\geq1}V^*Q_n\pi_n(a)Q_nV=V^*(\sum_{n\geq1}Q_n\pi_n(a)Q_n)V=V^*(\oplus_{n\geq1}\pi_n(a))V=V^*\pi(a)V=\phi(a),
\end{equation*}
for all $a\in\A$. The required assertion that $\phi$ is in BW-closure of $\cst$-convex hull of $\cst$-extreme points of $S_\h(\A)$ now follows from Lemma \ref{sum of pure maps is in C*-convex hull}.

\textit{Case II:} Let $\A$ be a type $I$ factor, say $\A=\bg$ for some Hilbert space $\G$. In view of Proposition \ref{UCP maps are approximated by normal UCP maps}, it  suffices to approximate a normal UCP map by $\cst$-convex combinations of $\cst$-extreme points of $ S_\h(\bg)$. Let $\phi:\bg\to\bh$ be a normal UCP map.  Then by Corollary \ref{expression for normal map as a sum}, there exists a finite or countable sequence of contractions $\{V_n\}_{n\geq1}$ in $\mathcal{B}(\h,\G)$ such that
\begin{equation}
    \phi(X)=\sum_{n\geq1}V_n^*XV_n\;\;\mbox{ for all }~ X\in\bg, \quad\text{(WOT Convergence)}.
\end{equation}
Note that the maps $X\mapsto V_n^*XV_n$ from $\bg$ to $\bh$ are pure maps. The claim now follows from Lemma \ref{sum of pure maps is in C*-convex hull}.
\end{proof}

\begin{remark}
In Case I of Theorem \ref{Krein-Milman type theorem} above, we have invoked a corollary of Voiculescu's result for representations on separable $\cst$-algebras acting on  separable Hilbert spaces. In recent years, there have been some study of Voiculescu's theorems beyond separable case by applications coming from logic to operator algebras (see Vaccaro \cite{Va}). The results of \cite{Va} are for separably acting representations on certain non-separable $\cst$-algebras, and are not directly applicable  in the current situation, as  the Hilbert space $\hpi$ in Case I of Theorem \ref{Krein-Milman type theorem} need not remain separable if $\A$ is not separable. Nevertheless, one can try to modify the proof above or look for  possible variations in Vaccaro's result  to extend our work beyond separable case.
\end{remark}

\section{Examples and applications}\label{section:examples and application}

In the final section, we discuss a number of examples of UCP maps with their $\cst$-extremity properties. We shall also see an application  to a well-known result from classical functional analysis about factorization property of Hardy algebras. We believe that the connection between $\cst$-extreme points and factorization property of the algebra $\M$ in Corollary \ref{a factorization property of algebras coming out of C* extreme point} will produce many more examples and applications.

First we look into the question of when tensor products of two $\cst$-extreme points are $\cst$-extreme. This will help us in producing more $\cst$-extreme points out of the existing ones. The tensor product in question is minimal tensor product. See \cite{Pa} for definitions and related properties.

For any two unital $\cst$-algebras $\A_1$ and $\A_2$, let $\A_1\otimes\A_2$ denote their minimal (or spatial) tensor product (when $\A_1,\A_2$ are von Neumann algebras, we denote by $\A_1\overline{\otimes}\A_2$ the von Neumann algebra generated by $\A_1\otimes\A_2$). Then for any two UCP maps $\phi_i:\A_i\to\B(\h_i)$, $i=1,2$, the assignment $a_1\otimes a_2\mapsto\phi_1(a_1)\otimes\phi_2(a_2)$, for $a_i\in\A_i$, extends to a UCP map from $\A_1\otimes\A_2$ to $\B(\h_1\otimes\h_2)$, which we denote by $\phi_1\otimes\phi_2$ (see Theorem 12.3, \cite{Pa}). The next proposition talks about $\cst$-extremity of tensor products, where one of the components is pure. We use the following well-known fact: if $\B_i\subseteq\B(\h_i), i=1,2,$ are two von Neumann algebras, then $(\B_1\otimes\B_2)'=\B_1'\overline{\otimes}\B_2'$ (Theorem IV.5.9, \cite{Ta}).

\begin{proposition}
Let $\phi_i:\A_i\to\B(\h_i)$, $i=1,2$, be two UCP maps, and let  $\phi_2$ be pure. Then $\phi_1$ is $\cst$-extreme (resp. extreme)  in $ S_{\h_1}(\A_1)$ if and only if  $\phi_1\otimes\phi_2$ is $\cst$-extreme (resp. extreme)  in $ S_{\h_1\otimes\h_2}(\A_1\otimes\A_2)$.
\end{proposition}
\begin{proof}
Let $(\pi_i,V_i,\K_i)$ be the minimal Stinespring triple of $\phi_i$ for $i=1,2$. Then it is immediate that $(\pi_1\otimes\pi_2, V_1\otimes V_2,  \K_1\otimes\K_2)$ is the minimal Stinespring triple for $\phi_1\otimes\phi_2$.
Set $\pi=\pi_1\otimes\pi_2$.  Note that since $\pi_2(\A_2)'=\C\cdot I_{\K_2}$ (as $\phi_2$ is pure), it follows from above mentioned result that \[\pi(\A)'=(\pi_1(\A_1)\otimes\pi_2(\A_2))'=\pi_1(\A)'\overline{\otimes }I_{\K_2}=\pi_1(\A)'\otimes I_{\K_2}.\]
Now for any operator $D=D_1\otimes I_{\K_2}\in\pi(\A)'$, we note that $D_1$ is positive and $V_1^*D_1V_1$ is invertible if and only if $D_1\otimes I_{\K_2}$ is positive  and $(V_1\otimes V_2)^*(D_1\otimes I_{\K_2})(V_1\otimes V_2)$ is invertible. Also $D_1(V_1\h_1)\subseteq V_1\h_1$ if and only if $(D_1\otimes I_{\K_2})(V_1\otimes V_2)(\h_1\otimes \h_2)\subseteq (V_1\otimes V_2)(\h_1\otimes\h_2)$. The assertion about equivalence of $\cst$-extreme points now follows from equivalent criteria in Corollary \ref{a C*-extreme criterian in language of factorization of positive operators}. The assertions about extreme points follow similarly using Extreme point condition (Theorem \ref{Extreme point condition}).
\end{proof}

Since the identity representation $\id_n:M_n\to M_n$ is pure, the following corollary  about ampliation of a $\cst$-extreme map is immediate.

\begin{corollary}
Let $\phi$ be a $\cst$-extreme point in $\sha$. Then the map $\phi\otimes \id_{n}:\A\otimes M_n\to\B(\h\otimes \C^n)$ is $\cst$-extreme in $S_{\h\otimes\C^n}(\A\otimes M_n)$, for each $n\in\N$.
\end{corollary}

For the next proposition, we set up some notations. Let $X$ be a
countable set. For any Hilbert space $\h$ and a von Neumann algebra
$\B\subseteq\bh$, we  consider the Hilbert space $\ell^2_\h(X)$ and
von Neumann algebra $\ell^\infty_{\B}(X)$ given by
\[\ell^2_\h(X)=\{f:X\to\h; \Sigma_{x\in X}\|f(x)\|^2<\infty\},
~\mbox{ and }~\ell^\infty_\B(X)=\{F:X\to\B; F \mbox{ is
bounded}\}.\] Then $\ell^\infty_\B(X)$ acts on the Hilbert space
$\ell^2_\h(X)$ via the operator $M_F,F\in\ell^\infty_\B(X)$, defined
by 
\[M_Ff(x)=F(x)f(x),~ \mbox{ for } f\in\ell^2_\h(X)~\mbox{ and }~ x\in X.\]
We  write
$\ell^2_\C(X)$ and $\ell^\infty_\C(X)$ simply by $\ell^2(X)$ and
$\ell^\infty(X)$ respectively. Also we identify the Hilbert space
$\ell_\h^2(X)$ with $\ell^2(X)\otimes \h$, and the algebra
$\ell^\infty_\B(X)$ with $\ell^\infty(X)\overline{\otimes}\B$, so
that we shall use them interchangeably. If there is no possibility
of confusion, we shall drop $X$ from $\ell^2(X), \ell^2_\h(X)$ etc.

\begin{proposition}\label{tensor product of phi and i}
Let $\phi$ be a $\cst$-extreme point in $\sha$, and let  $i:\ell^\infty(X)\to\B(\ell^2(X))$ be the natural inclusion map for some countable set $X$. Then $i\otimes \phi$ is $\cst$-extreme  in $S_{\ell^2\otimes\h}(\ell^\infty\otimes\A)$.
\end{proposition}
\begin{proof}
Let $(\pi, V,\hpi)$ be the minimal Stinespring triple for $\phi$. Then
$(\rho, U,\h_{\rho})$ is the minimal Stinespring triple for $i\otimes \phi$, where $\h_{\rho}=\ell^2\otimes\hpi=\ell^2_{\hpi}$, $U=i\otimes V:\ell^2\otimes\h\to\ell^2\otimes\hpi$, and $\rho=i\otimes\pi$. As mentioned above, we have
\[\rho(\ell^\infty\otimes\A)'=(\ell^\infty\otimes\pi(\A))'=\ell^\infty\overline{\otimes}\copa=\ell^\infty_{\copa}.\]
Now let $M_D\in \ell^\infty_{\copa}$ be a positive operator such that $ U ^*M_D U $ is invertible. Then there exists $ \alpha  >0$  such that $ U ^*M_D U \geq\alpha U^*U$.  Note that for any $f\in \ell^2_{\h}$ and $x\in X$, we have
\begin{equation*}
     U ^*M_D U f(x)=(V^*D(x)V)f(x).
\end{equation*}
Therefore for any unit vectors  $g\in \ell^2$ and $h\in\h$, we have
\begin{equation*}
    \alpha\leq \langle U^*M_DU(g\otimes h), g\otimes h\rangle=\sum_{x\in X}\langle (V^*D(x)V)g(x)h, g(x)h\rangle =\sum_{x\in X} \langle (V^*D(x)V)h, h\rangle \;|g(x)|^2,
\end{equation*}
 and since $g\in \ell^2$ varies over all unit vectors, it follows (by choosing $g$ to be the canonical basis elements of $\ell^2$) that $\langle(V^*D(x)V)h,h\rangle\geq \alpha$ for all $x\in X.$  Again since $h\in \h$ is arbitrary, it follows that $V^*D(x)V\geq \alpha$ for all $x\in X$, i.e. $V^*D(x)V$ is invertible in $\bh$. Since $\phi$ is $\cst$-extreme in $\sha$,  there exists  an operator  $S(x)\in\pi(\A)'$ for each $x\in X$, such that $D(x)=S(x)^*S(x)$, $S(x)VV^*=VV^*S(x)VV^*$ and $V^*S(x)V$ is invertible. Also note  that 
 \[\|(V^*S(x)V)^{-1}\|^2=\|(V^*D(x)V)^{-1}\|\leq 1/\alpha.\]
 If $S$ denotes the map $x\mapsto S(x)$ from $X$ to $\copa$, then it is immediate to verify that $S\in\ell^\infty_{\copa}$ such that $M_D={M_S}^*M_S$ and $M_SUU^*=UU^*M_SUU^*$. Also since $\sup_{x\in X}\|(V^*S(x)V)^{-1}\|^2\leq 1/\alpha$, it follows that $ U ^*M_S U $ is invertible. Since $M_D$ is arbitrary, we conclude that  $i\otimes\phi$ is $\cst$-extreme.
\end{proof}

If the set $X$ in Proposition \ref{tensor product of phi and i} is a two point set, then we get the following:

\begin{corollary}
Let $\phi$ be a $\cst$-extreme point in $\sha$. Then the map $\psi:\A\oplus\A\to\B(\h\oplus\h)$ defined by $\psi(a\oplus b)=\phi(a)\oplus\phi(b)$, for all $a,b\in\A$, is a $\cst$-extreme point in $S_{\h\oplus\h}(\A\oplus\A)$.
\end{corollary}

 The next proposition provides a family of $\cst$-extreme points, which can be thought as a generalization of Example 2 in \cite{FaMo}, and whose proof follows almost the same lines. We give the proof for the sake of completeness. For doing so, we need the following fact from $\cst$-convexity of unit ball of $\bh$ (see  \cite{HoMoPa} for definitions and Theorem 1.1 therein):  all isometries and co-isometries are  $\cst$-extreme points of closed unit ball of $\bh$.

 We also use the  following assertion which is easy to verify  (also see Theorem 3.18, \cite{Pa}): if $(\pi,V,\hpi)$ is the minimal Stinespring triple for a UCP   map $\phi\in\sha$, then for any $a\in\A$, $\phi(a)^*\phi(a)=\phi(a^*a)$ if and only if $V\phi(a)=\pi(a)V$.

 Below, $C^*(T)$ denotes the unital $\cst$-algebra generated by an operator $T$.

\begin{proposition}\label{isometries gives C*-extreme}
Let $S$ be a unitary, and let $\phi:C^*(S)\to\bh$ be a UCP map such that $\phi(S)$ is an isometry or a co-isometry. Then $\phi$ is $\cst$-extreme as well as extreme in $S_{\h}(C^*(S))$.
\end{proposition}
\begin{proof}
We assume that  $\phi(S)$ is an  isometry. The case of  $\phi(S)$ a co-isometry follows similarly. Let $(\pi,V,\hpi)$ be the minimal Stinespring triple for $\phi$. Since $\phi(S)$ is an isometry, we have $\phi(S)^*\phi(S)=I_\h=\phi(1)=\phi(  S^*  S)$,  so  it follows (as mentioned above) that $V\phi(S)= \pi(S)V.$ This in particular implies for each $n\in\mathbb{N}$ that $V\phi(S)^n= \pi(S)^nV$, which yields
\begin{equation}\label{phi^n(S)=phi(S^n)}
     \phi(S)^n=V^*\pi(S)^nV=V^*\pi(S^n)V=\phi(S^n).
\end{equation}
Now to prove that $\phi$ is $\cst$-extreme in $S_\h(C^*(S))$, let $\phi=\sum_{i=1}^nT_i^*\phi_i(\cdot)T_i$ be a proper $C^*$-convex combination for some UCP maps $\phi_i$ and invertible operators $T_i\in\bh$ with $\sum_{i=1}^nT_i^*T_i=I_\h$. Since  $\phi(S)$ is an isometry, it is a $C^*$-extreme point in the closed unit ball of $\bh$ (Theorem 1.1, \cite{HoMoPa}); hence  there exist unitaries $U_i\in\bh$ satisfying
\[\phi(S)=U_i^*\phi_i(S)U_i\]
for each $i$.
This implies that each $\phi_i(S)$ is an isometry, and in a similar fashion as in \eqref{phi^n(S)=phi(S^n)}, we get
\begin{equation}
    \phi_i(S)^n=\phi_i(S^n)~~\mbox{ for all }n\in\mathbb{N}.
\end{equation}
Thus for each $n\in\mathbb{N}$, we have
\begin{equation*}
    \phi(S^n)=\phi(S)^n=(U_i^*\phi_i(S)U_i)^n=U_i^*\phi_i(S)^nU_i=U_i^*\phi_i(S^n)U_i.
\end{equation*}
By taking adjoint both the sides, we also get  $\phi({S^*}^n)=U_i^*\phi_i({S^*}^n)U_i.$ Since $S$ is unitary, it follows that $\overline{\Span}\{S^n, {S^*}^m; n,m\in\N\} =C^*(S)$. Thus we conclude that
 $\phi(T)=U_i^*\phi_i(T)U_i$ for every $T\in C^*(S)$ i.e. $\phi$ is unitarily equivalent to $\phi_i$. The case of $\phi$ being extreme follows on similar lines, as isometries and co-isometries are extreme points of the closed unit ball of $\bh$.
\end{proof}

As a special case of Proposition \ref{isometries gives C*-extreme}, we have the following result.
Here $z\in C(\T)$ is the function on the unit circle $\T$  given by $z(e^{i\theta})=e^{i\theta}$ for $\theta\in\R$.

\begin{corollary}\label{isometry or coisometry on CT gives C*-extreme points}
Let $\phi:C(\mathbb{T})\to \bh$ be a UCP map such that $\phi(z)$ is an isometry or a co-isometry. Then $\phi$ is  $C^*$-extreme  as well as  extreme  in $S_\h(C(\T))$.
\end{corollary}

As an application of Corollary \ref{isometry or coisometry on CT gives C*-extreme points}, we give a new and simplified proof of a  classical result of Szeg\"o and its operator valued  analogue about factorization property of Hardy algebras. Let $\K$ be a Hilbert space (possibly infinite dimensional), and let $L^2_\K(\T)$ denote the Hilbert space of  $\K$-valued square integrable functions on $\T$ with respect to one-dimensional Lebesgue measure (which is isomorphic to $L^2(\T)\otimes\K$). Let $ H^2_\K(\T)$ denote the subspace 
\[\{f\in L^2_\K(\T); \int_0^{2\pi}f(e^{i\theta})e^{-in\theta}d\theta=0\mbox{ for all }n<0\}\]
    of $L^2_\K(\T)$  (called {\em vector-valued Hardy space}). Let $L^\infty_{\B(\K)}(\T)$ be the von Neumann algebra of all essentially bounded measurable functions from $\T$ to $\mathcal{B}(\K)$,  which acts on $L^2_\K(\T)$ by left multiplication i.e. for $F\in L^\infty_{\bk}(\T)$, the operator $M_F:L^2_\K(\T)\to L^2_\K(\T)$ is defined by 
\[M_Ff(x)=F(x)f(x)\quad\mbox{ for all}~~ f\in L^2_\K(\T), x\in\T.\]
Let $H^\infty_{\B(\K)}(\T)$ be its subalgebra defined by
\[H^\infty_{\B(\K)}(\T)=\{F\in  L^\infty_{\B(\K)}(\T);\int_0^{2\pi}F(e^{i\theta})e^{-in\theta}d\theta=0 \mbox{ for all }n<0\}.\]
The algebra $H^\infty_{\bk}(\T)$ is called the {\em operator-valued Hardy algebra}. Note that $C(\T)\subseteq L^\infty(\T)\subseteq L^\infty_{\bk}(\T)$. We have the following factorization property of $H^\infty_{\bk}(\T)$ in $L^\infty_{\bk}(\T)$.

\begin{corollary}\label{factorization property of hardy algebra}
For any positive and invertible $D\in L^\infty_{\mathcal{B}(\K)}(\T)$, there exists an invertible $S$ with $S, S^{-1}\in H^\infty_{\B(\K)}(\T)$ such that $D=S^*S$.
\end{corollary}
\begin{proof}
Consider the UCP map $\phi:C(\T)\to \B( H^2_\K(\T))$ defined by
\begin{equation}\label{UCP map on hardy space}
    \phi(f)=P_{ H^2_\K(\T)}{M_f}_{|_{ H^2_\K(\T)}},~\mbox{ for all }f\in C(\T).
\end{equation}
Clearly $\phi(z)$ is an isometry, so it follows from  Corollary \ref{isometry or coisometry on CT gives C*-extreme points} that $\phi$ is a $\cst$-extreme point in $ S_{ H^2_\K(\T)}(C(\T))$. Note that the map $\phi$ is already in minimal Stinespring form, where the representation $\pi$ acts on the Hilbert space $L^2_\K(\T)$ by $\pi(f)=M_f$, for all $f\in C(\T)$. It is well-known that  $\pi(C(\T))'=L^\infty_{\B(\K)}(\T)$ (Theorem 52.8, \cite{Co}), and it is easy to verify that
\[ H^\infty_{\B(\K)}(\T)=\{F\in L^\infty_{\B(\K)}(\T); M_F( H^2_\K(\T))\subseteq H^2_\K(\T)\}.\]
The required assertion now follows from Corollary \ref{a factorization property of algebras coming out of C* extreme point}.
\end{proof}

\begin{example}
Let $T\in\bh$ be an isometry or a co-isometry. Consider the linear map $\phi:C(\T)\to\bh$ satisfying $\phi(p+\bar{q})=p(T)+q(T)^*$ for polynomials $p$ and $q$. Then $\phi$ extends to a UCP map on $C(\T)$ (Theorem 2.6, \cite{Pa}), and it follows from Corollary \ref{isometry or coisometry on CT gives C*-extreme points} that $\phi$ is $\cst$-extreme as well as extreme in $S_{\h}(C(\T))$.
\end{example}

Following is an example $\phi$ of a $\cst$-extreme map of $S_\h(C(\T))$ such that $\phi(z)$ need not be an isometry or a co-isometry.

\begin{example}
Let $g:\T\to\T$ be a homeomorphism, and let $\phi:C(\T)\to\bh$ be a UCP  map. Set $\psi:C(\T)\to\bh$ by $\psi(f)=\phi(f\circ g)$ for  all $f\in C(\T)$. Then it is easy to verify that $\phi$ is $\cst$-extreme in $S_\h(C(\T))$ if and only if $\psi$ is $\cst$-extreme in $S_\h(C(\T))$. Moreover one can  choose a homeomorphism $f$ such that $\phi(z)$ is an isometry but $\psi(z)$ is neither an isometry nor a co-isometry.
\end{example}

The following are two examples of  (normal) UCP maps which are not  $\cst$-extreme points. In order to show this, we use the fact that  nest algebras associated with uncountable complete nests  do not have factorization.

\begin{example}\label{example of non C*-extreme point using indexing set of rational}
Let $\K$ be a Hilbert space, and let $\{\K_q\}_{q\in\Q}$ be a nest of subspaces indexed by rationals $\Q$ such that  $\K_q\subsetneq\K_{q'}$  if $q< q'$, and $\K=\vee_{q\in\Q}\K_q$ . Let $\G$ be a Hilbert space, and let $\{\G_q\}_{q\in\Q}$ be any collection of mutually orthogonal subspaces of $\G$. Consider the subspace $\h=\oplus_{q\in\Q}\G_q\otimes\K_q$ of $\G\otimes\K$, and  the map $\phi:\bg\to\bh$ defined by \[\phi(X)=P_\h(X\otimes I_\K)_{|_{\h}},~~\mbox{ for all } X\in\bg.\]
Note that the algebra $\M=\{T\in\bk; (I_\G\otimes T)(\h)\subseteq\h\}$ is nothing but $\Alg\E$, where $\E$ is the nest $\E=\{\K_q\}_{q\in\Q}$. Even though the nest $\E$ is  countable, its completion is not a countable nest (indeed, completion of $\E$ is given by $\{0, \K, \K_q, \L_r; q\in\Q, r\in\R\}$ where $\L_r=\bigvee_{p<r}\K_p$); so it follows from Theorem \ref{Larson's factorization result} that $\M$ does not have factorization in $\bk$. Consequently, $I_\G\otimes\M$ does not have factorization in $I_\G\otimes\bk=\copa$, where $\pi(X)=X\otimes I_\K$ is the minimal Stinespring representation of $\phi$. Thus we conclude from Corollary \ref{a factorization property of algebras coming out of C* extreme point} that $\phi$ is not a $\cst$-extreme point in $ S_\h(\bg)$.
\end{example}

\begin{example}
Let $\K=L^2([0,1])$ with respect to Lebesgue measure, and let $\h=\{\chi_\Delta f; f\in L^2([0,1]\times [0,1])\}\subseteq \K\otimes\K$, where $\Delta=\{(s,t); s,t\in[0,1], 0\leq s\leq t\leq1\}\subseteq[0,1]\times[0,1]$. Here $\chi_\Delta$ denotes the characteristic function on the set $\Delta$. Define $\phi:\bk\to\bh$ by
\begin{equation*}
    \phi(X)=P_\h(X\otimes I_\K)_{|_{\h}}~~\mbox{for all }X\in\bk.
\end{equation*}
We claim that $\phi$ is not a $\cst$-extreme point in $S_\h(\bk)$. First consider the following observations, which  are straightforward to verify:
\begin{itemize}
    \item $\h=\ospan\{\chi_{[0,t]}f\otimes\chi_{[t,1]}g; t\in [0,1], f,g\in\K\}$.

\item $\h^\perp=\ospan\{\chi_{[s,1]}f\otimes \chi_{[0,s]}g; s\in [0,1], f,g\in\K\}$.

\item $\K\otimes\K=\ospan\{(X\otimes I_\K )h; h\in \h,X\in\bk\}$.

\item $\phi(X)=P_\h\pi(X)_{|_\h}$ is the minimal Stinespring dilation for $\phi$ where $\pi:\bk\to\B(\K\otimes\K)$ is defined by $\pi(X)=X\otimes I_\K $, $X\in\bk$.

\item $\pi(\bk)'=\{I_\K \otimes S; S\in\bk\}$.
\end{itemize}
Let $\M=\{ S\in\bk; (I_\K\otimes S)(\h)\subseteq\h\}$. We claim that $\M\subseteq  \Alg\E$, for the complete nest
$\E=\{E_t; t\in [0,1]\}$, where
\[E_t=\{\chi_{[t,1]}f; f\in \K\}, \mbox{ for } t\in [0,1].\]
 Since  $\E$ is uncountable, it will follow from Theorem \ref{Larson's factorization result} that $\Alg\E$ does not have factorization in $\bk$; hence  $\M$ does not have factorization in $\bk$, that is, $I_\K\otimes\M$ does not have factorization in $I_\K\otimes\bk=\pi(\bk)'$.  This will imply from Corollary \ref{a factorization property of algebras coming out of C* extreme point} that $\phi$ is not $\cst$-extreme in $S_\h(\bk)$.

Now let $S\in\M$, so that $(I_\K \otimes S)(\h)\subseteq \h$.
Fix $t\in (0,1]$, and let $0<s<t$. Note that $E_s^\perp=\{\chi_{[0,s]}f; f\in \K\}$. Now for any $f,g\in \K$, we note from above observations that $ \chi_{[0,t]}\otimes \chi_{[t,1]}g\in \h$ (so that $(I_\K \otimes S)( \chi_{[0,t]}\otimes \chi_{[t,1]}g)\in \h$) and $\chi_{[s,1]}\otimes \chi_{[0,s]}f\in \h^\perp$; hence
\begin{equation*}
\begin{split}
    0&=\left\langle (I_\K \otimes S)( \chi_{[0,t]}\otimes \chi_{[t,1]}g), \chi_{[s,1]}\otimes \chi_{[0,s]}f\right\rangle=\left\langle  \chi_{[0,t]}\otimes S(\chi_{[t,1]}g), \chi_{[s,1]}\otimes \chi_{[0,s]}f\right\rangle\\
    &=\langle \chi_{[0,t]}, \chi_{[s,1]}\rangle \langle S(\chi_{[t,1]}g), \chi_{[0,s]}f\rangle= (t-s)\langle S(\chi_{[t,1]}g), \chi_{[0,s]}f\rangle.
\end{split}
\end{equation*}
Since $t-s\neq0$, it follows that $\langle S(\chi_{[t,1]}g), \chi_{[0,s]}f\rangle=0.$ This shows that $S(\chi_{[t,1]}g)\perp E_s^\perp$, which is to say $S(\chi_{[t,1]}g)\in E_s$. Since $g\in \K$ is arbitrary, it follows that  $S( E_t )\subseteq E_s$. Since $s<t$ is arbitrary, we conclude that
\begin{equation*}
    S( E_t )\subseteq \bigcap_{0<s<t}E_s= E_t .
\end{equation*}
This shows  that $S\in\Alg\E$; thus we conclude our claim that $\M\subseteq\Alg\E$.
\end{example}

Inspired from the example of $\cst$-extreme point as in \eqref{eq:Hardy space example}, we now consider its noncommutative analogue.
For a $\cst$-subalgebra $\A$ of $\bk$ and a subspace $\h$ of $\K$, consider the UCP map $\phi:\A\to\bh$ given by \[\phi(X)=P_\h X_{|_{\h}},\quad\mbox{for}~~ X\in\A.\]
If $\A=\bk$, then clearly $\phi$ is a pure map, so that $\phi$ is $\cst$-extreme in $\sha$. An example of $\cst$-extreme point of this form (when $\A\neq \bk$) is the map in \eqref{UCP map on hardy space}. But for arbitrary $\A$, we do not know if $\phi$ is always $\cst$-extreme in $ S_\h(\A)$.

Let $\A$ be a finite von Neumann algebra with a distinguished
faithful trace $\tau$. Let $L^2(\tau)$ denote the  Hilbert space
induced by $\tau$, which is the  closure of $\A$ with respect to the
inner product on $\A$ defined by $\langle x,y\rangle=\tau(x^*y)$ for
$x,y\in\A$. Then the left regular representation
$\pi:\A\to\B(L^2(\tau))$ defined by $\pi(x)=L_x$ for all $x\in\A$, is
cyclic with cyclic vector $\delta=1$,  where $L_x:L^2(\tau)\to
L^2(\tau)$ is  given by
\[L_x(y)=xy,\quad\mbox{ for all }~y\in\A.\]
Now let $\M$
be a subalgebra of $\A$ such that $\M$ has factorization in $\A$ (as
defined in \ref{definition of algebras with factorization}).
Examples of such algebras are finite maximal subdiagonal algebras
introduced by Arveson \cite{Ar67}, which also include nest
subalgebras. Consider the subspace $ H^2=[\M]\subseteq L^2(\tau)$
(called {\em noncommutative Hardy space}), and let $\phi:\A\to\B( H^2)$ be
the map defined by
\[\phi(x)=P_{ H^2}{L_x}_{|_{ H^2}},\]
for $x\in\A$.  It is clear that $\phi$ is a UCP map. We have the following:

\begin{proposition}
For $\A,\M$ and $\phi$  as above,  $\phi$ is a $\cst$-extreme point in $ S_{ H^2}(\A)$.
\end{proposition}
\begin{proof}
Note that  $(\pi,V,L^2(\tau))$ is the minimal Stinespring triple, where $V$ is the inclusion map from $H^2$ to $L^2(\tau)$. It is a well-known fact that $\pi(\A)'=\{R_x; x\in\A\}$ (see Proposition 11.16, \cite{Pi}), where $R_x\in\B(L^2(\tau))$ is the right multiplication operator defined by $R_x(y)=yx$ for all $y\in\A$.

Now to show that $\phi$ is $\cst$-extreme in $S_{H^2}(\A)$, we let $R_x$ to be a positive operator in $\pi(\A)'$ for some $x\in\A$ such that $P_{ H^2}{R_x}_{|{ H^2}}$ is invertible. Clearly $x\geq 0$ in $\A$. We claim that $x$ is invertible in $\A$. Since $P_{ H^2}{R_x}_{|_{ H^2}}$ is invertible, there is an $\alpha>0$  such that $P_{ H^2}{R_x}_{|_{ H^2}}\geq \alpha I_{ H^2}$. Hence for all $z\in\M$, we have  $\langle zx,z\rangle=\langle R_xz,z\rangle\geq \alpha\langle z,z\rangle$, that is, $\tau((x-\alpha)z^*z)=\langle z(x-\alpha),z\rangle\geq 0$. Since $\{z^*z;z\in\M\}$ is dense in the set  of all positive elements of $\A$ (as $\M$ has factorization in $\A$), it follows that $\tau((x-\alpha)y)\geq0$, for all $y\geq 0$ in $\A$. Hence for all $a\in\A$, we get using the trace property of $\tau$ that
\[\langle (x-\alpha)a,a\rangle=\tau(a^*(x-\alpha)a)=\tau((x-\alpha)aa^*)\geq0,\]
which is to say that $x-\alpha\geq0$ in $\A$. This shows that $x$ is invertible. Therefore by factorization of $\M$ in $\A$, there exists an invertible element $z$ with $z,z^{-1}\in\M$ such that $x=zz^*$; thus $R_x=R_{zz^*}=R_{z^*}R_z=R_z^*R_z$. Further, since $z\in\M$, it follows that $R_z(\M)\subseteq\M$ and hence $R_z( H^2)\subseteq H^2$. Also since $z^{-1}\in\M$, we have $R_z^{-1}( H^2)=R_{z^{-1}}( H^2)\subseteq H^2$, which in particular implies that ${R_z}_{|_{ H^2}}$ is invertible. Since $R_x$ is arbitrary in $\copa$, we conclude that $\phi$ is a $\cst$-extreme point in $ S_{ H^2}(\A)$.
\end{proof}

\begin{acknowledgment*}
The first author thanks J C Bose Fellowship (SERB, India) for financial support. We sincerely thank the referee for several constructive suggestions which helped us to improve the paper.
\end{acknowledgment*}

\end{document}